\documentclass[11pt]{amsart}
\usepackage[utf8]{inputenc}
\usepackage{lmodern}
\usepackage{amsmath, amsthm, amssymb, amsfonts,bm}
\usepackage[normalem]{ulem}
\usepackage{hyperref}

\usepackage{verbatim} 
\usepackage{longtable}

\usepackage{mathtools}

\usepackage{tikz}
\usetikzlibrary{decorations.pathmorphing}
\tikzset{snake it/.style={decorate, decoration=snake}}
\usepackage{tikz-cd}
\usetikzlibrary{positioning}

\usepackage{caption}

\usepackage{tikz-cd}
\usetikzlibrary{arrows}

\theoremstyle{plain}

\newtheorem{theorem}{Theorem}[section]
\newtheorem{conj}[theorem]{Conjecture}
\newtheorem{definition}[theorem]{Definition}
\newtheorem{prop}[theorem]{Proposition}

\newtheorem{cor}[theorem]{Corollary}
\newtheorem{rem}[theorem]{Remark}
\newtheorem{lem}[theorem]{Lemma}

\DeclareFontFamily{OT1}{rsfs}{}
\DeclareFontShape{OT1}{rsfs}{n}{it}{<-> rsfs10}{}
\DeclareMathAlphabet{\curly}{OT1}{rsfs}{n}{it}

\def\Q{\mathbb{Q}}

\def\Z{\mathbb{Z}}

\def\C{\mathbb{C}}

\def\P{\mathbb{P}}

\def\O{\mathcal{O}}

\def\a{\alpha}
\def\b{\beta}
\def\gm{\gamma}

\def\m{\mu}

\def\sm{\Sigma}

\def\m{\mathcal}
\def\f{\mathfrak}

\def\bb{\mathbb}

\def\td{\widetilde}

\def\tx{\textrm}
\def\Spec{\textrm{Spec}}

\def\ch{\textrm{ch}}

\def\Td{\textrm{Td}}
\def\Id{\textrm{Id}}
\def\rk{\textrm{rank}}

\def\ker{\textrm{ker}}

\def\gr{\textrm{gr}}
\def\Pic{\textrm{Pic}}
\def\Hom{\textrm{Hom}}
\def\Ext{\textrm{Ext}}
\def\res{\textrm{res}}
\def\Tr{\textrm{Tr}}
\def\para{\textrm{par-}}

\def\uxi{{\underline{\xi}}}
\def\huxi{\mathcal{H}_{\underline{\xi}}}
\def\muxi{\mathcal{M}_{\underline{\xi}}}
\def\zuxi{Z_{\underline{\xi}}}
\def\suxi{S_{\underline{\xi}}}
\def\euxi{\mathcal{E}_{\underline{\xi}}}
\def\euxiflagd{\mathcal{E}^\bullet_{\huxi\times D}: 0=\mathcal{E}^0_{\huxi\times D} \subset...\subset \mathcal{E}^l_{\huxi\times D}= (\mathcal{E}_{\uxi})|_{\huxi\times D}}
\def\euxiflagp{\mathcal{E}^\bullet_{\huxi\times p}: 0=\mathcal{E}^0_{\huxi\times p} \subset...\subset \mathcal{E}^l_{\huxi\times p}= (\mathcal{E}_{\uxi})|_{\huxi\times p}}
\def\quxi{Q_{\underline{\xi}}}
\def\pure{\textrm{pure}}
\def\um{{\underline{m}}}
\def\ua{{\underline{\alpha}}}
\def\nld{\mathcal{N}(l,D)}

\def\nldo{\mathcal{N}(1,D)}
\def\relativeh{\bm{\mathcal{H}}( C, D;r, d,\underline{\a}, \underline{m})}
\def\relativeho{\bm{\mathcal{H}}( C, D;1, d, \underline{m})}
\def\frelativeh{\mathfrak{H}(C,D;r,d,\underline{\a}, \underline{m})}
\def\frelativehxi{\mathfrak{H}_{\underline{\xi}}(C,D;r,d,\underline{\a}, \underline{m})}
\def\dE{\mathcal{E}}
\def\mparabun{\mathcal{M}(C,D;r,d,\underline{\alpha},\underline{m})}
\def\fmsheaves{\mathfrak{M}(S_{\underline{\xi}}; (0,\Sigma_{\underline{m}},c),\b, A) }

\newsavebox{\leftbox} \newsavebox{\rightbox}%

\NewDocumentCommand{\lrboxbrace}{s O{\{} O{\}} O{0.05\linewidth} m O{0.5\linewidth} m}{% \lrboxbrace[<lbrace>][<rbrace>][<lwidth>]{<ltext>}[<rwidth>]{<rtext>}
  \begin{lrbox}{\leftbox}% Left box
    \IfBooleanTF{#1}% starred/unstarred
      {\begin{varwidth}{#4}#5\end{varwidth}}
      {\begin{minipage}{#4}#5\end{minipage}}
  \end{lrbox}
  \begin{lrbox}{\rightbox}% Right box
    \IfBooleanTF{#1}% starred/unstarred
      {\begin{varwidth}{#6}#7\end{varwidth}}
      {\begin{minipage}{#6}#7\end{minipage}}
  \end{lrbox}
  \ensuremath{\usebox\leftbox\left\{\,\usebox\rightbox\,\right\}}
}

\usepackage[backend=biber,style=alphabetic,sorting=nyt, giveninits=true,doi=false,isbn=false,maxbibnames=99]{biblatex} %main package for bibtex
\usepackage{vmargin}
\setpapersize{A4}
\setmarginsrb{27mm}{12mm}{27mm}{12mm}%
{12mm}{10mm}{5mm}{10mm}

\usepackage{tikz}
\usepackage{lmodern}
\usetikzlibrary{decorations.pathmorphing}

\begin{document}
\title{Generators for the cohomology of the moduli space of irregular parabolic Higgs bundles}

\author{Jia Choon Lee and Sukjoo Lee}
\address{Peking University, Beijing International Center for Mathematical Research, Jingchunyuan Courtyard \#78, 5 Yiheyuan Road, Haidian District,
Beijing 100871, China}
\email{jiachoonlee@pku.edu.cn}

\address{Department of Mathematics, University of Edinburgh, EH9 3FD, UK}
\email{Sukjoo.Lee@ed.ac.uk}

\begin{abstract}
We prove that the pure part of the cohomology ring of the moduli space of irregular $\underline{\xi}$-parabolic Higgs bundles is generated by the K\"{u}nneth components of the Chern classes of a universal bundle and the Chern classes of the successive quotients of a universal flag of subbundles. As an application, in the regular full-flag case, we demonstrate a similar result for the cohomology ring of the moduli spaces of parabolic and strongly parabolic Higgs bundles. 
\end{abstract}

\baselineskip=14.5pt
\maketitle

\setcounter{tocdepth}{2}

\tableofcontents
\section{Introduction}
It is well-known that the cohomology ring of the moduli space of stable sheaves on certain varieties $X$ can be generated by the tautological classes i.e. the K\"{u}nneth components of the Chern classes of a 
universal sheaf: Atiyah-Bott \cite{atiyah-bott} when $X$ is a curve, Ellingsrud-Stromme \cite{ellingsrud-stromme} when $X=\P^2$, Beauville \cite{beauville} when $X$ is a rational or ruled surface, Markman \cite{markman2002symplectic} when $X$ is a symplectic surface, Markman \cite{markman2006integral} when $X$ is a Poisson surface. For the moduli space of stable Higgs bundles on a curve, similar results are found in the work of Markman \cite{markman2002symplectic}\cite{markman2006integral}. On the other hand, for the moduli space of stable parabolic bundles on a curve, it is also known that the tautological classes obtained from a universal bundle alone are not enough to generate the cohomology ring, one needs to take into account of some extra classes coming from a universal flag of subbundles (see Biswas-Raghavendra \cite{biswas-raghavendra}).

In this paper, we consider the cohomology ring of the moduli space of stable parabolic Higgs bundles on a curve with a regular(=tame) or irregular(=wild) singularity and a fixed polar part for the Higgs fields. Historically, parabolic Higgs bundles with regular singularities were considered by Simpson \cite{simpson-noncompact} in order to generalize the non-abelian Hodge correspondence to punctured curves. The next natural generalization is to consider Higgs bundles on a curve with irregular singularities. According to the wild non-abelian Hodge correspondence of Sabbah \cite{sabbah} and Biquard-Boalch \cite{biquardboalch2004}, fixing an equivalence class of the polar part, the moduli space of stable parabolic Higgs bundles with irregular singularities is diffeomorphic (by a hyper-K\"{a}hler rotation) to the moduli space of stable parabolic connections with irregular singularities. Furthermore, the moduli spaces of Higgs bundles or connections on a curve with irregular singularites provide a wide class of  interesting examples of hyper-K\"{a}hler manifolds which are related to classical integrable systems
(see \cite{boalch2012hyperkahler}).

One of the motivations for considering the cohomology of the moduli spaces of parabolic Higgs bundles with irregularity singularities is inspired by the $P=W$ conjecture due to de Cataldo-Hausel-Migliorini \cite{dCMH12}. When there is no (regular or irregular) singularity on the Higgs field, the conjecture states that the perverse Leray filtration on the cohomology of the moduli space of Higgs bundles (Dolbeault side) is identified with the weight filtration of the cohomology of the corresponding character variety (Betti side) via the non-abelian Hodge correspondence. There are now several approaches to the $P=W$ conjectures (for $GL_n$) due to Maulik-Shen \cite{maulik2022pw}, Hausel-Mellit-Minets-Schiffmann \cite{hausel2022pw}, Maulik-Shen-Yin \cite{maulik2023perverse}. One can ask the same question when there are regular or irregular singularities (call it the wild $P=W$ conjecture) since the corresponding Dolbeault and Betti moduli spaces are related by the wild non-abelian Hodge correspondence. There are some works in this direction for special cases: Shen-Zhang \cite{shen-zhang21pw} for five families of parabolic Higgs bundles with regular singularity on $\P^1$, Szab\'{o} \cite{szabo21pw}\cite{Szabo23pw} for low dimensional moduli spaces of rank 2 Higgs bundles on $\P^1$ with (regular or irregular) singularities. 

A common and key ingredient for all the different approaches \cite{ShenMaulik2022},\cite{hausel2022pw},\cite{maulik2023perverse} of the $P=W$ conjecture is the above-mentioned theorem of Markman \cite{markman2002symplectic} about the generation result of the cohomology of the moduli space of Higgs bundles on a curve by tautological classes. The Chern filtration spanned by the tautological classes plays the role as an intermediate filtration in establishing the equality between the perverse Leray and weight filtrations on both sides. Therefore, in order to approach the wild $P=W$ conjecture, a natural first step will be to understand the generators of the cohomology ring of the moduli space of parabolic Higgs bundles with singularities. 

However, with the presence of parabolic structures, we need to include the Chern classes of the successive quotients of a universal flag of subbundles as in the case of moduli of parabolic bundles \cite{biswas-raghavendra}. These classes are known to be important from the viewpoint of representation theory e.g. in the global Springer theory of Yun \cite{yun-global-springer}. Furthermore, even in situations where the cohomology ring of the moduli space of Higgs bundles without parabolic structures are mainly concerned, it is often useful to first pass to the cohomology ring of the moduli of parabolic Higgs bundles  with regular singularities (without fixing polar parts) e.g. \cite{ShenMaulik2022}, \cite{hausel2022pw}, \cite{maulik2023perverse}.

\subsection{Main results}
Let $C$ be a smooth projective curve of genus $g\geq 0$ and $D=np$ supported at a point $p$ and $n\geq 1.$ Fix the numerical data: $r\geq 1, d\in \Z$, $1>\a_1> \dots> \a_l\geq 0$ (each $\a_i\in \Q)$, $m_1,\dots,m_l\in \Z\geq 0$ such that $\sum_{i=1}^l m_i=r$. An irregular parabolic Higgs bundle with a pole of order $n\geq 1$ at $p$ is a quadruple $(E, E^\bullet_D, \Phi, \ua)$ where $E$ is a rank $r$, degree $d$ vector bundle on $C$, $E^\bullet_D: 0= E^0_D\subset E^1_D\subset \dots \subset E^{l-1}_D\subset E^l_D=E|_D$ is a quasi-parabolic structure with $\chi(E^i_D/E^{i-1}_D)=nm_i$, $\Phi:E\to E\otimes K_C(D)$ is a Higgs field, $\ua=(\a_1,\dots,\a_l)$ is the set of parabolic weights, such that $\Phi_D(E^i_D)\subset E^{i}_D\otimes K_C(D)|_D$. When the order of the pole is 1 (i.e. $n=1$), we call it a regular parabolic Higgs bundle.

In order to fix the polar part of the Higgs field, we choose $\uxi=(\xi_1,\dots, \xi_l)$ where $\xi_i\in H^0(D, K_C(D)|_D).$ An irregular (resp. regular) $\uxi$-parabolic Higgs bundles is an irregular (resp. regular) parabolic Higgs bundle $(E, E^\bullet_D, \Phi, \ua)$ satisfying the following conditions: (1) $E^i_D/E^{i-1}_D$ is free of rank $m_i$ on $D$, (2) $\gr_i(\Phi_D) = \Id _{E^i_D/E^{i-1}_D}\otimes\xi_i$ for $i=1,\dots,l$.

We first show that there exists a coarse moduli space of stable irregular $\uxi$-parabolic Higgs bundles on $C$. In fact, we show the existence of a relative coarse moduli space of stable irregular $\uxi$-parabolic Higgs bundles on $C$ when $\uxi$ is varied over a base $\nld:=\{ \uxi= (\xi_1,...,\xi_l)| \xi_i\in H^0(D, K_C(D)|_D) ,  \sum \res (\xi_i)= 0 \}$. 
\begin{theorem}[Theorem \ref{existence-of-coarse-moduli}]\label{thm:intro-existence} 
    Fix the numerical data: $r\geq 1, d\in \Z$, $1>\a_1> ...> \a_l\geq 0$ where $\a_i\in \Q$, $m_1,...,m_l\in \Z_{\geq 0}$. There exists a relative coarse moduli scheme $P: \relativeh\to \nld$ of stable irregular $\uxi$-parabolic Higgs bundles.    \end{theorem}

We denote by $\huxi:=P^{-1}(\uxi)$ the moduli space of stable irregular $\uxi$-parabolic Higgs bundles for a fixed $\uxi\in \nld.$ Suppose that $\uxi$ is generic i.e. the leading terms of the $\xi_i$ are all distinct. Suppose $r$ and $d$ are coprime so that there exist a universal bundle $\m E_\uxi$ on $\huxi\times C$ and a universal flag of subbundles $\euxiflagd$. We denote by $\euxiflagp$ the  restriction of $\m E_{\huxi\times D}^\bullet$ to $\huxi\times \{p\}$ and $\quxi^i:= \m E^i_{\huxi\times p} /\m E^i_{\huxi\times p}$ the successive quotients of the flag $\m E_{\huxi\times p}^\bullet$. The goal of this paper is to show that the K\"{u}nneth components of the Chern classes of $\euxi$ and the Chern classes of $Q^i_\uxi$ generate the pure part of the cohomology $H^*(\huxi, \Q)$. 

Our strategy to study the generators of the cohomology ring follows closely the approach of Markman in the case of moduli space of Higgs bundles on a curve \cite{markman2002symplectic}\cite{markman2006integral}. In our situation, the first and key step is to employ the spectral correspondence due to Kontsevich-Soibelman \cite{Kontsevich-Soibelman} and Diaconescu-Donagi-Pantev \cite{Diaconescu_2018}. There is a similar spectral correspondence due to Szab\'{o} \cite{Szab2017TheBG} which works for an open subset of the moduli space, we will mainly follow \cite{Diaconescu_2018} since we need the correspondence for the whole moduli space. When $\uxi$ is generic, the spectral correspondence realizes $\huxi$ in terms of the moduli space of $\b$-twisted $A$-Gieseker stable pure dimension one sheaves on a (non-compact) holomorphic symplectic surface $\suxi$ for a suitable choice of $\b , A\in NS(\zuxi)_{\Q}$ with $A$ ample, where $\zuxi$ is a natural compactification of $\suxi$. The natural compactification $\zuxi$ of $\suxi$ also provides a modular compactification of $\huxi$ by the moduli space $\muxi$ of $\b$-twisted $A$-Gieseker stable pure dimension one sheaves on $\zuxi$. Then we apply the argument of Markman to express the diagonal class of $\muxi\times \huxi$ in terms of the Chern classes of a universal sheaf on $\muxi\times \zuxi.$ By applying the Grothendieck-Riemann-Roch Theorem, we in turn express these classes in terms of the K\"{u}nneth components of the Chern classes of $\euxi$ and the Chern classes of $\quxi^i$.

\begin{theorem}[Theorem \ref{thm:purgen}]\label{intro-purgen}
        Let $\uxi$ be generic. The pure cohomology $H^*_{pure}(\huxi, \Q)$ is generated by the K\"{u}nneth components of the Chern classes of $\euxi$ and the Chern classes of $\quxi^i$, where $1\leq i\leq l.$
    \end{theorem}

Moreover, since the symplectic surface $\suxi$ is constructed by a sequence of blow-ups on the total space $\textrm{Tot}(K_C(D))$ followed by the removal of a divisor, we can see that the Chern classes of $\quxi^i$ arises naturally from the classes of the exceptional divisors. 
 
Specializing to the regular $(n=1)$ and full-flag $(\um = (1,\dots,1))$ case, let $\m H:= \bm{\mathcal{H}}( C, p;r, d,\underline{\a}, (1,\dots, 1))$ and $\huxi= P^{-1}(\uxi)$ as before. We also denote by $\m M$ (resp. $\m M_0$) the coarse moduli space of stable regular parabolic (resp. strongly parabolic) Higgs bundles constructed by Yokogawa \cite{Yokogawa-compactification}. In this case, every regular parabolic Higgs bundle of the full-flag type is automatically a regular $\uxi$-parabolic Higgs bundle for a unique $\uxi$, so  we have $\m H\cong \m M$ (Proposition \ref{prop:isomodulispace}). Then we show that there is an isomorphism of $\Q$-mixed Hodge structures $H^*(\mathcal{H},\Q) \cong H^*(\huxi,\Q)$ for any $\uxi$ and they are of \textit{pure type}. This follows from the property of semi-projectivity of $\mathcal{H}$ which will be proved in Section \ref{sec:purity}. We note that this method has been used to show the similar result for the moduli space of Higgs bundles \cite{HauselRodriguez2013}. As an application of Theorem \ref{intro-purgen}, we have the following result. 

\begin{cor}[Corollary \ref{cor:generator reg par}]
    The cohomology ring of the coarse moduli space of stable regular parabolic (resp. strongly parabolic) Higgs bundles of the full-flag type $\mathcal{M}$ (resp. $\mathcal{M}_0$) is generated by the Künneth components of the Chern classes of a universal bundle and the Chern classes of the successive quotients of a universal flag of subbundles. 
\end{cor}
A similar result can be found in Oblomkov-Yun \cite[Theorem 3.1.8]{oblomkov2017cohomology} by a different approach.

All the results can be directly generalized to the case where there are more than one irreducible component in $D$, namely when $D=n_1p_1+\cdots+n_kp_k$ for $k>1$ and $n_j\geq 1$ for $j=1, \dots, k$. The moduli problem we study is easily generalized by putting the independent $\uxi_j$-parabolic conditions for each $j$. For the spectral correspondence, one can construct the holomorphic symplectic surface $S_{\uxi_1, \dots, \uxi_k}$ by simultaneously blowing-up the locus over each $p_j$. Moreover, this gives rise to a universal flag of subbundles corresponding to each $p_j$. Therefore, in this case, the generators in our generation result will be the Künneth components of the Chern classes of a universal bundle and the Chern classes of the successive quotients of a universal flag of subbundles corresponding to each $p_j$, where $j=1,\dots,k$. 

\subsection{Outline}
In Section 2, we discuss several moduli problems associated with irregular parabolic Higgs bundles and prove Theorem \ref{thm:intro-existence}. Additionally, we review the spectral correspondence whose proof will be in Appendix \ref{proof of spectral}. This will serve as a key argument in the next section. In Section 3, we study the generators of the pure part of the cohomology ring of the moduli space of stable $\uxi$-irregular parabolic Higgs bundles, applying the method used in Markman's works \cite{markman2002symplectic, markman2006integral}. As an application, we focus on the regular full-flag case and demonstrate the generators of the cohomology ring of the moduli space of regular (strongly) parabolic Higgs bundles. 

\subsection{Notations}\quad \\
\noindent \textbf{Conventions}: For a fixed scheme $X$ and $T\in \textrm{Sch}$ (or Sch/$B$ for any base scheme $B$), we will denote by $X_T:= X\times T$. If $T=\Spec(A)$ is affine, we also write $X_A:= X\times \Spec(A)$. For  a divisor $D$ in $X$ and  a coherent sheaf $F$ on $X$, we write $F_D= F\otimes_{\O_X}\O_D$.

\noindent \textbf{Notations}: Here we summarize the notations for various moduli spaces used in this paper. 
Fix $C$ to be a smooth projective curve, $p\in C$, $D=np$ for $n\geq 1$, $M=K_C(D)$, $r$ rank, $d$ degree,  $\underline{\alpha}$ parabolic weights, $\underline{m}=(m_1, \dots, m_l)$ where $\sum m_i=r$.
\begin{itemize}
    \item $\nld:=\{ \uxi= (\xi_1,...,\xi_l)| \xi_i\in H^0(D, M_D) ,  \sum \res (\xi_i)= 0 \}.$ 
    \item $P:\relativeh \to \nld$, the relative coarse moduli space of stable irregular $\uxi$-parabolic Higgs bundles. 
    \item $\huxi:=P^{-1}(\uxi)$, the coarse moduli space of stable irregular $\uxi$-parabolic Higgs bundles for $\uxi \in \nld$.
    \item $\mathcal{M}(C, D;r, d, \underline{\a}, \underline{m})$, the coarse moduli space of stable irregular parabolic Higgs bundles. 
    \item $\mathcal{M}_0(C, D;r, d, \underline{\a}, \underline{m})$, the coarse moduli space of stable irregular strongly parabolic Higgs bundles. 
    \item $G:\relativeh\to \mathcal{M}(C, D;r, d, \underline{\a}, \underline{m})$, the induced morphism introduced in Remark \ref{rem:comparison fails}.
\end{itemize}
In the regular ($n=1$) full-flag ($\underline{m}=\underline{1}=(1, \cdots, 1)$) case, we simply write 
\begin{itemize}
    \item $\mathcal{H}={\bf{\mathcal{H}}}(C,p;r,d,\underline{\alpha}, \underline{1})$ and $\mathcal{N}=\mathcal{N}(l,p)$.
    \item $\mathcal{M}=\mathcal{M}(C, p;r, d, \underline{\a}, \underline{1})$ and $\mathcal{M}_0=\mathcal{M}_0(C, p;r, d, \underline{\a}, \underline{1})$.
    
\end{itemize}

\section{Moduli spaces}
\subsection{Irregular parabolic Higgs bundles}
Let $C$ be smooth projective curve of genus $g\geq 0$ and $D= np$ supported at a point $p$ and $n \geq1 $. Let $M=K_C(D)$ be a twisted canonical line bundle. 

\begin{definition}
   An irregular parabolic Higgs bundle on $C$ with a pole of order $n\geq 1$ at $p$ consists of the following data: 
    
    \begin{enumerate}
        \item A Higgs bundle $(E, \Phi)$ where $\Phi: E\to E\otimes M$
        \item A quasi-parabolic structure:
        \begin{equation*}
            E^\bullet_D:  0 = E^0_D\subset E^1_D \subset ... \subset E^{l-1}_D\subset E^l_D = E_D
        \end{equation*}
        such that $\Phi_D(E^i_D) \subset E^i_D\otimes_D M_D$.

        \item A collection of parabolic weights $\underline{\a}= (\a_1,...,\a_l)\in \Q^l$: 
        \begin{equation*}
            1> \a_1 > \a_2>...>\a_l\geq 0
        \end{equation*}
        \end{enumerate}
        \end{definition} 
        For simplicity, we will only consider the case with a single pole in this paper, we simply call it an irregular parabolic Higgs bundle and denote it by $(E, E^\bullet_D,\Phi,\ua)$.
         When the order of the pole is 1 (i.e. $n=1$), we call it a regular parabolic Higgs bundle. When $\Phi_D(E^i_D) \subset E^{i-1}_D\otimes_D M_D$ for all $i$ in condition $(2)$, we call it a strongly (regular or irregular) parabolic Higgs bundle. 
         
    \begin{rem}\label{Yokogawa's definition}
        The more general definition of parabolic Higgs bundles used in \cite{Yokogawa-compactification} is the so-called parabolic $\Omega$-pairs on any smooth, projective, geometrically integral, locally noetherian scheme $X/S$, and $D$ a relative effective Cartier divisor on $X/S$, where $\Omega$ is any locally free sheaf. Suppose $S=\Spec(\C), X=C, D=np, \Omega = K_C(D)$. A parabolic $\Omega$-parabolic sheaf on $C$ consists of a vector bundle $E$, a filtration of vector bundles
        \begin{equation*}
            F^\bullet(E): E(-D)= F^0(E)\subset F^1(E)\subset ... \subset F^l(E)= E,
        \end{equation*}
        a collection of parabolic weights $1> \a_1> \a_2>\dots > \a_l\geq 0$ and a homomorphism $\Phi: E\to E\otimes \Omega$ such that $\Phi(F^i(E))\subset F^i(E)$. Given $E^\bullet_D$ as in the definition of a parabolic Higgs bundle, we can recover $F^\bullet(E)$ as follows. Set $F^i(E) := \ker(E\to E_D\to E_D/E^{i}_D)$. As $E^{i-1}_D\subset E^i_D$, the map $E\to E_D/E^i_D$ factors through $E_D/E^{i-1}_D$ and so $F^{i-1}(E)\subset F^i(E)$ as required.  Conversely, given $F^\bullet(E)$, we define $E^i_D:= F^i(E)/F^0(E).$ Therefore, a parabolic $\Omega$-parabolic sheaf on $C$ is equivalent to a parabolic Higgs bundle on $C$ in our definition.

    \end{rem}

     In order to consider the coarse moduli space for such objects, we will need to discuss the stability condition.
     Recall that in \cite[Section 1]{Maruyama-Yokogawa} and \cite[Section 1]{Yokogawa-compactification},  the parabolic Euler characteristic and the (reduced) parabolic Hilbert polynomial of $E_*:= (E, E_D^\bullet,\underline{\a}) $ are defined as 
    \begin{equation*}
        \para \chi({E_*}) = \chi(E )+\sum_{i=1}^l \a_i\chi(E^i_D/E^{i-1}_D) , \qquad \para P_{E_*}(t) = \frac{\para \chi(E_*(t))}{\rk(E)}
    \end{equation*}
    where $E^i_D/E^{i-1}_D$ is viewed as a torsion sheaf on $C$ in the expression and $E_*(t) = (E\otimes \O(t), E^\bullet_D\otimes \O(t),\ua).$  \footnote{In the original definition, $\para \chi({E_*}) := \chi(E(-D) )+\sum_{i=1}^l \a_i\chi(E^i_D/E^{i-1}_D)$, but the difference will not affect the stability condition in the case of curves. }

    \begin{definition}
    An irregular parabolic Higgs bundle $(E_*, \Phi) $ is said to be $\underline{\a}$-(semi)stable if for any nontrivial proper subbundle $0\subset F\subset E$ preserved by $\Phi$, we have 
    \begin{equation*}
        \para P_{F_*}(t) < \para P_{E_*}(t)  \quad \textrm{for } t\gg 0 \quad (\textrm{resp.  }\leq )
    \end{equation*}
    where $F_*= (F, F^\bullet_D, \underline{\a})$ is defined by the induced filtration $F^i_D= F_D\cap E^i_D$, for $i\leq l$, that is preserved by $\Phi_D$. 
        
    \end{definition}

    The work of Yokogawa \cite[Theorem 4.6]{Yokogawa-compactification} shows that there exists a coarse moduli space for $\underline{\a}$-stable irregular parabolic Higgs bundles of rank $r$, degree $d$ such that $\chi(E^i_D/E^{i-1}_D)=nm_i$. We will denote this moduli space by $\mparabun$. Moreover, one can define the coarse moduli space of $\underline{\alpha}$-stable irregular strongly parabolic Higgs bundles as a closed subscheme in $\mathcal{M}(C,D;r, d, \underline{\a}, \underline{m})$, which we denote by $\mathcal{M}_0(C,D;r, d, \underline{\a}, \underline{m}).$ 

    In order to fix the polar part $\Phi_D$ of the Higgs fields, we choose sections $\xi_1, \dots, \xi_l\in H^0(D, M_D)$ and denote by $\underline{\xi}=(\xi_1, \dots, \xi_l)$ the collection of such sections.  Following \cite{Diaconescu_2018}, we call a quadruple $(E ,E^\bullet_D, \Phi,\underline{\alpha})$ \textit{an irregular $\underline{\xi}$-parabolic Higgs bundle} if (1) the successive quotients $E^i_D/E^{i-1}_D$ of $E^\bullet_D$ are free $\O_D$-modules (hence all $E^i_D$ are also free) and (2) the induced morphism of $\O_D$-modules $\textrm{gr}_i\Phi_D:= \Phi_{D, i}: E^i_D/E^{i-1}_D\to E^i_D/E^{i-1}_D \otimes M_D$ satisfies the following condition
    \begin{equation}\label{xi-parabolic}
        \Phi_{D, i}= \mathrm{Id}_{E^i_D/E^{i-1}_D}\otimes \xi_i, \quad 1\leq i\leq l.
    \end{equation} 

    \begin{rem}
        The definition of irregular $\uxi$-parabolic Higgs bundle used here and \cite{Diaconescu_2018} is a direct analogue of the notion of unramified irregular singular $\bm{\nu}$-parabolic connection of parabolic depth $n$ introduced in the work of Inaba-Saito \cite[Definition 2.1]{inaba-saito}, where $\bm{\nu}=\uxi$ in our case. The following Theorem \ref{existence-of-coarse-moduli} and Proposition \ref{smoothness-of-det} are also the analogues of \cite[Theorem 2.1]{inaba-saito} and \cite[Proposition 2.1]{inaba-saito}, respectively.
    \end{rem}

    \begin{rem}
        Note that the condition that the successive quotients of $E^\bullet_D$ are free implies that there exists a filtration of vector spaces $E^\bullet_p: 0=E^0_p\subset E^1_p \subset \dots \subset E^{l-1}_p \subset E^l_p= E_p$ such that $E^\bullet_D = E^\bullet_p\otimes \O_D$. 
    \end{rem}

    More generally, there exists a relative moduli space of irregular $\uxi$-parabolic Higgs bundles on $C$ when $\uxi$ varies. 
    Let $\nld = \{ \uxi= (\xi_1,...,\xi_l)| \xi_i\in H^0(D, M_D) ,  \sum \res (\xi_i)= 0 \}$ be the parameter space of the polar parts $\uxi$. Define the moduli functor 
    $\frelativeh : \textrm{Sch}/\nld\to \textrm{Sets}$ as follows: for each $T\in \textrm{Sch}/\nld$ represented by $(\uxi)_T= ((\xi_1)_T,...,(\xi_l)_T)\in \nld(T)$, we have 
    \begin{center}
    $\frelativeh (T)=$
  \lrboxbrace{ } 
  {vector bundles $\m E$ on $C_T$ such that $\m E|_{C_t}$ is locally free of rank $r$ and degree $d$ for each geometric point $t\in T$, 
         Higgs fields $\Psi: \m E \to \m E\otimes  \Omega_{C_T/T}^1(D_T)$,
        quasi-parabolic structures $\m E_{D_T}^\bullet: 0 = \m E_{D_T}^0 \subset \m E_{D_T}^1\subset \dots \subset \m E^l_{D_T}=\m E_{D_T}$ such that $\Psi_{D_T}(\m E^i_{D_T})\subset \m E^i_{D_T}\otimes  \Omega^1_{C_T/T}(D_T)$ and 
        each $\m E^i_{D_T}/\m E^{i-1}_{D_T}$ is locally free of rank $m_i$ on $D_T$, $\gr_i(\Psi_{ D_T})-\Id_{\m E^i_{D_T}/\m E^{i-1}_{D_T}}\otimes (\xi_i)_T=0$, and $(\m E|_{C_t}, \m E^\bullet_{D_T}|_{D_t},\Psi|_{C_t}, \underline{\a})$ is $\underline{\a}$-stable for each geometric point $t\in T$.}$\bigg/\sim$
\end{center}
where two flat families are equivalent $(\m E, \m E^\bullet_{D_T},\Psi) \sim (\m E', \m E'^\bullet_{D_T},\Psi')$  if there exists a line bundle $L$ on $T$ such that $(\m E, \m E^\bullet_{D_T},\Psi)\cong (\m E'\otimes p_2^*L, \m E'^\bullet_{D\times T}\otimes p_2^*L,\Psi'\otimes p_2^*L)$ and $p_2:C_T=C\times T\to T$ is the projection.
    \begin{theorem}\label{existence-of-coarse-moduli}
        Fix the numerical data: $r\geq 1, d\in \Z$, $1>\a_1> ...> \a_l\geq 0$ where $\a_i\in \Q$, $m_1,...,m_l\in \Z_{\geq 0}$. There exists a relative coarse moduli scheme $P: \relativeh\to \nld$ of the moduli functor $\frelativeh$.    \end{theorem}
    \begin{proof}
          As explained in Remark \ref{Yokogawa's definition}, we can view an  irregular $\uxi$-parabolic Higgs bundle as a "parabolic $\Omega$-pair" in the sense of \cite{Yokogawa-compactification}, where $\Omega= K_C(D)$ in our case. The extra conditions imposed here are the freeness of $E^i_D/E^{i-1}_D$ and the condition \eqref{xi-parabolic} on the polar part of the Higgs fields.

    By definition, a family of parabolic $\Omega$-pairs consists of a triple $({\m E}, F^\bullet({\m E}), {\Psi})$ where ${\m E}$ is a vector bundle on $ C_T$, $F^\bullet({\m E}): \m E(-D_T)= F^0({\m E})\subset F^1({\m E}) \cdots \subset F^l({\m E})= {\m E}$ is a filtration of vector bundles such that ${\m E}/F^i({\m E})$ is flat over $T$ for $0\leq i\leq l$, and ${\Psi}: {\m E}\to {\m E} \otimes \Omega_{ C_T/T}^1( D_T)$ such that ${\Psi}(F^i({\m E}))\subset F^i({\m E})\otimes \Omega_{ C_T/T}^1(D_T)$. Since $\m E/F^i(\m E)$ is flat over $T$ for $0\leq i\leq l$, it follows from the short exact sequence $0\to F^i(\m E)/F^{i-1}(\m E)\to \m E/F^{i-1}(\m E)\to \m E/F^{i}(\m E)\to 0$ that the quotient $F^i(\m E)/F^{i-1}(\m E)$ is flat over $T$ for $1\leq i\leq l$.  Since $F^i(\m E)/F^{i-1}(\m E)\cong F^i(\m E)_{ D_T}/F^{i-1}(\m E)_{ D_T}$, the flatness implies that the subset 
    $\{t\in T| F^i(\m E)_{ D_t}/F^{i-1}(\m E)_{ D_t}\textrm{  is locally free}\}$ is open in $T$ \cite[Lemma 5.4]{newstead}. Therefore, the condition of the freeness of $E^i_D/E^{i-1}_D$ is an open condition. On the other hand, for each $T\in \textrm{Sch}/\nld$ represented by $(\uxi)_T= ((\xi_1)_T,...,(\xi_l)_T)\in \nld(T)$, the condition \eqref{xi-parabolic} becomes the vanishing of the homomorphisms $\gr_i({\Psi}_{C_T}|_{D_T})-\Id_{F^i(\m E)_{ D_T}/F^{i-1}(\m E)_{ D_T}}\otimes (\xi_i)_T.$ Then we can always find a closed $\nld$-subscheme $T'\subset T$ with the corresponding universal property (see \cite[Corollary 2.3]{Yokogawa-compactification} for example). Therefore, the condition \eqref{xi-parabolic} on the polar part is a closed condition.      
   
   In \cite[Theorem 4.6]{Yokogawa-compactification}, the coarse moduli scheme $ \mathcal{M}(C, D;r, d, \underline{\alpha}, \underline{m})$ is constructed as a GIT quotient of a parameter scheme $R$ with a group action of $PGL(V)$ for some $k$-vector space $V$. There is a universal family of stable parabolic $\Omega$-pairs $(\m G,F^\bullet(\m G), \bm{\Psi})$ over $C_R$ and a surjection $V\otimes_k \O_{C_R} \twoheadrightarrow \m G$. Consider $\mathcal{M}(C, D;r, d, \underline{\alpha}, \underline{m})\times \nld$ as a scheme over $\nld$. As explained above, we can first restrict to an open subscheme $R'$ of $R$ defined by the freeness condition. Then there is a closed subscheme $R''\subset R'\times \nld$ defined by condition \eqref{xi-parabolic}.  It is clear that $R''$ is $PGL(V)$-invariant as the group only acts on the surjections $V\otimes \O_{ C}\twoheadrightarrow  E$ in the parameter scheme. Hence, the GIT quotient of $R''$ by $PGL(V)$ will be a locally closed subscheme $\relativeh$ (over $\nld$) of $\mathcal{M}(C, D;r, d, \underline{\alpha}, \underline{m})\times \nld$ satisfying the property of a coarse moduli scheme for $\frelativeh$. 
    \end{proof}

    \begin{rem}
        In particular,  the fiber $P^{-1}(\uxi)$ over a fixed polar part $\uxi\in \nld$ will be the coarse moduli scheme for stable irregular $\uxi$-parabolic Higgs bundles.  

    \end{rem}
        
    \begin{prop}[Regular full flag case]\label{prop:isomodulispace}
    In the regular full-flag case, equivalently $D=p$ and $\underline{m}= (1,..., 1)$, we have $\relativeh \cong \mathcal{M}(C, D;r, d, \underline{\alpha}, \underline{m})$. 
    \end{prop}
    \begin{proof}
        In this case, the freeness condition is automatic since $\O_D\cong k$, so every $\O_D$-module must be free. A quasi-parabolic structure is a filtration of vector spaces $0=E^0_p\subset E^1_p\subset \cdots E^l_p= E|_p$ and each $E^i_p/E^{i-1}_p$ is one dimensional. So for each $1\leq i\leq l$, the induced map $\Phi_{p,i}= \gr_i(\Phi_p): E^i_p/E^{i-1}_p\to E^i_p/E^{i-1}_p\otimes M_p$ must be a scalar multiplication i.e. $\Phi_{p,i} = \Id\otimes \xi_i$ for some unique $\xi_i\in H^0(p,M_p)$. Therefore, a regular parabolic Higgs bundle is a regular $\uxi$-parabolic Higgs bundle for a unique $\uxi\in \m N(l,D).$  Then it is clear that $\relativeh\cong \mathcal{M}(C, D;r, d, \underline{\alpha}, \underline{m})$. 
    \end{proof}

\begin{rem}\label{rem:comparison fails}
    In general, there is a morphism $ G:\relativeh  \to \mathcal{M}(C, D;r, d, \underline{\alpha}, \underline{m})$ which is given by the composition 
    \[
    \relativeh \hookrightarrow \mathcal{M}(C, D;r, d, \underline{\alpha}, \underline{m})\times \nld \to \mathcal{M}(C, D;r, d, \underline{\alpha}, \underline{m}).
    \]
    However, this morphism $G$ is not surjective: when $n>1$, there exists a filtration $E_D^\bullet$ whose successive quotients are not necessarily free. This implies that $G$ is not necessarily surjective. See the discussion after Corollary \ref{cor:generator reg par} as well. When $\underline{m}$ does not represent the full-flag type, there exists a Higgs field $\Phi$ which is not necessarily diagonal when restricted to $D$. 
\end{rem}

    Next, we study the smoothness of the relative moduli space $P:\relativeh \to \nld$. Note that when $r=1$, the quasi-parabolic structure on $E_D$ is trivial, so we can write the relative coarse moduli space as $\relativeho\to \nldo$ which consists of a line bundle $L$, $\Phi\in H^0(C, M)$ such that $\Phi_D = \Id_ {L_D} \otimes \xi $ for some $\xi\in H^0(D,M_D)$. There is a determinant morphism $\relativeh \to \relativeho$ which sends $(E,E_D^\bullet,\Phi) $ to $(\det(E), \det(\Phi))$ where $\det(\Phi) := (\Phi\wedge \Id \wedge\dots \wedge\Id) +(\Id\wedge \Phi\wedge \dots \wedge\Id)+ \dots + (\Id\wedge \Id \wedge \dots \wedge\Phi)$. Define the trace map $\Tr: \nld \to \nldo$ by $\Tr(\uxi) = \sum_{i=1}^l \xi_i$. Then it is easy to check that $P(\det(E),\det(\Phi)) = \Tr(P(E,E^\bullet_D,\Phi))$. So the two maps induce $\det: \relativeh \to \relativeho\times _{\nldo} \nld.$

        \begin{prop}\label{smoothness-of-det}
            The morphism $\det: \relativeh \to \relativeho\times _{\nldo} \nld$ is smooth. 
        \end{prop}    
        
        \begin{proof}
        The proof here is essentially a modification of \cite[Proposition 2.1]{inaba-saito} to the Higgs bundles case. It suffices to show that the morphism of moduli functors $\det: \frelativeh \to \relativeho\times _{\nldo} \nld$ is formally smooth. 
        Let $A$ be an Artinian local $k$-algebra with maximal ideal $\f m_A$ and residue field $k=A/\f m_A.$ Let $0\to I\to A'\to A\to 0$ be a small extension i.e. $\f m_{A'}I=0.$ We shall show that a lift $\Spec(A')\to \frelativeh$ always exists in each of the following diagram
        \begin{equation}\label{lifting-diagram}
            \begin{tikzcd}
                \Spec(A)\arrow[r,"q"]\arrow[d] & \frelativeh\arrow[d,"\det"]\\ 
                \Spec(A')\arrow[ru, dotted] \arrow[r,"p"]& \relativeho\times _{\nldo} \nld
            \end{tikzcd}
        \end{equation}
Let $((L,\Phi^L), \uxi)\in \relativeho \times _{\nldo} \nld(A')$ corresponding to the morphism $p$ such that $\Phi^{L} = \Id_{L|_D}\otimes (\sum^l_{i=1}\xi_i)$. Let $(\dE,\dE_{D_A}^\bullet, \Phi)\in \frelativeh(A)$ corresponding to $q$ such that $(\det(\dE) ,\det(\Phi))\cong (L,\Phi^L) \otimes_{A'} A$ and $P(E, E^\bullet_{D_A},\Phi) = \uxi \otimes_{A'}A$. Denote by $(\overline{E},\overline{E}_{D}^\bullet, \overline{\Phi}):=(\dE,\dE_{D_A}^\bullet, \Phi)\otimes k$.

        Choose a \v{C}ech cover $\{U_\a\}$ of $C$  which trivializes $\dE$ i.e. $\dE|_{U_\a\otimes A} \cong \O^{\oplus r}_{U_\a\otimes A}$ over each open subset $U_\a\otimes A\subset C_A$. The strategy is to first construct a local object over each $U_\a\otimes A'$ lifting $((L, \Phi^L),\uxi)$ and then study the obstruction for the existence of a global object. First, we take a free $\O_{U_\a\otimes A'}$-module $\dE_\a$ with isomorphisms $\varphi_\a: \det(\dE_\a)\xrightarrow{\sim} L|_{U_\a\otimes A'}$ and $\phi_\a:\dE_\a \otimes_{A'} A\xrightarrow{\sim} \dE|_{U_a\otimes A}$ such that $\varphi_\a\otimes A = \det(\phi_\a).$ If $p\in U_\a$, we can choose a basis $( e_j)_j$ of $\dE_\a$ (also a basis for $\dE_\a\otimes_{A'}A)$ such that $\dE^\bullet_{D_A}$ corresponds, via $\phi_\a$, to the standard filtration associated to the basis. Then the basis of $\dE_\a$ determines a quasi-parabolic structure  $(\dE_{\a})^\bullet_{D_{A'}}$. With respect to $(e_j)_j$, $\Phi_\a:= \phi_\a^{-1}\circ \Phi \circ \phi_{\a}$ is a matrix valued in $\O(U_\a\otimes A)$ such that
        \begin{equation*}
            \Phi_\a|_{D_A} = \begin{pmatrix}
                (\xi_1 \otimes _{A'}A)\otimes \Id_{\O_{D_A}^{\oplus m_1}} &\ast &\cdots &\ast\\
                0&(\xi_2\otimes _{A'}A)\otimes \Id_{\O_{D_A}^{\oplus m_2}} &\cdots &\vdots\\
                \vdots&\vdots&\ddots&\ast\\
                0&0&\cdots& (\xi_l\otimes _{A'}A)\otimes \Id_{\O_{D_A}^{\oplus m_l}}
            \end{pmatrix} 
        \end{equation*}
        Since each entry is an element of $\O(U_\a\otimes A)$, we can find a lift $\Phi'_\a$ of the matrix $\Phi_\a$ which is a matrix valued in $\O(U_\a\otimes A')$ such that 
        \begin{equation*}
            \Phi'_\a|_{D_{A'}} =\begin{pmatrix}
                \xi_1\otimes \Id_{\O_{D_{A'}}^{\oplus m_1}} &\ast &\cdots &\ast\\
                0&\xi_2\otimes \Id_{\O_{D_{A'}}^{\oplus m_2}} &\cdots &\vdots\\
                \vdots&\vdots&\ddots&\ast\\
                0&0&\cdots& \xi_l\otimes \Id_{\O_{D_{A'}}^{\oplus m_l}}
            \end{pmatrix} 
        \end{equation*}
        and such that $\det(\Phi'_\a)= (\varphi_\a \otimes \Id)^{-1}\circ \Phi^L \circ \varphi_\a$. Therefore, we get a local irregular $\uxi$-parabolic Higgs bundle $(\dE_\a, (\dE_\a)^\bullet_{D_{A'}}, {\Phi_\a})$ over $U_\a\otimes A'$. If $p\not\in U_{\a}$, then we take a lift $\Phi'_\a$ of $\Phi_\a$ valued in $\O(U_\a\otimes A')$ such that  $\det(\Phi'_A)= (\varphi_\a \otimes \Id)\circ \Phi^L \circ \varphi_\a^{-1}$ and a quasi-parabolic structure is not needed for this open subset. 

        Then we shall argue that the obstruction to glue the local objects $(\dE_\a, (\dE_\a)^\bullet_{D_{A'}},\Phi_\a)$ is given by a class in $\bb{H}^2$ of the following complex:
        \begin{equation*}
            \m D^\bullet_0: \quad 0\to PEnd_0(\overline{E}^\bullet)\xrightarrow{[\overline{\Phi},\cdot ]} SPEnd_0(\overline{E}^\bullet)\otimes K_C(D)\to 0
        \end{equation*}
        where 
        \begin{align*}
            PEnd_0(\overline{E}^\bullet)&= \{ f\in \m End(\overline{E})|\quad  f|_{D}(\overline{E}^i_{D})\subset (\overline{E}^i_{D}), \Tr(f)=0, \quad \forall i \}\\
            SPEnd_0(\overline{E}^\bullet)\otimes K_C(D)&=  \{ f\in \m End(\overline{E})|\quad  f|_{D}(\overline{E}^i_{D})\subset (\overline{E}^{i-1}_{D})\otimes M_D, \Tr(f)=0, \quad \forall i \}
        \end{align*}

        Denote by $U_{\a\b}:= U_\a\cap U_\b$ and $U_{\a\b\gm}:= U_\a\cap U_\b\cap U_\gm$. 
        For each pair of $\a,\b$, we choose a lift $\theta_{\b\a}: \dE_\a|_{U_{\a\b}}\xrightarrow{\sim} \dE_\b|_{U_{\a\b}}$ of the transition function $\phi_\b^{-1}\circ \phi_\a$, i.e. $\theta_{\b\a}\otimes A= \phi_\b^{-1}\circ \phi_\a$, such that $\varphi_\b\circ \det(\theta_{\a\b} )=  \varphi_\a$. Then we define
        \begin{align*}
            u_{\a\b\gamma} &:= \phi_\a\circ( \theta_{\gm\a}^{-1}\circ\theta _{\gm\b}\circ \theta_{\b\a} -\Id ) \circ \phi^{-1}_{\a} ; \\
            v_{\a\b}&:= \phi_\a\circ (\Phi'_{\a}-  \theta_{\a\b}^{-1} \circ \Phi'_{\b} \circ \theta_{\a\b})\circ \phi^{-1}_\a.
        \end{align*}
        One can check that $u_{\a\b\gm}\in \Gamma(U_{\a\b\gm}, PEnd_0(\overline{E}^\bullet))\otimes I$ because $u_{\a\b\gamma}$ is an endomorphism of parabolic bundles $(\dE|_{U_{\a\b\gm}\otimes A'}, \dE_{D_{A'}}^\bullet)$ with fixed determinant whose restriction to $A$ vanishes. Similarly, we have $v_{\a\b}\in \Gamma(U_{\a\b}, SPEnd_0(\overline{E}^\bullet)\otimes M) \otimes I.$ Hence, we have
        \begin{equation*}
            \{u_{\a\b\gm}\} \in C^2( \{U_{\a}\}, PEnd_0(\overline{E}^\bullet))\otimes I, \quad \{v_{\a\b}\} \in C^1( \{U_{\a}\}, SPEnd_0(\overline{E}^\bullet))\otimes I. 
        \end{equation*}
        It can be checked that $\{u_{\a\b\gm}\}$ and $\{v_{\a\b}\}$ uniquely define an obstruction class $\omega(\{u_{\a\b\gm}\}, \{v_{\a\b}\})\in \bb{H}^2(\m D^\bullet_0)$. Moreover, it is clear from the construction that $\omega(\{u_{\a\b\gm}\}, \{v_{\a\b}\})$  vanishes if and only if the local objects $\{(\dE_\a, (\dE_\a)^\bullet _{D_{A'}}, \Phi_\a)\}$ can be glued to $(\td{\dE}, \td{\dE}^\bullet _{D_{A'}}, \td{\Phi})\in \frelativeh (A')$ such that $\det(\td{\dE}, \td{\dE}^\bullet _{D_{A'}}, \td{\Phi})= ((L,\Phi^L),\uxi)$. Therefore, the lifting of $p$ in the diagram \ref{lifting-diagram} is equivalent to the vanishing of $\omega(\{u_{\a\b\gm}\}, \{v_{\a\b}\}). $

        By Serre duality, $\bb{H}^2(\m D_0^\bullet)\cong \bb{H}^0(\check{\m D}_0^\bullet)^\vee$ where $\check{\m D}^\bullet_0$ is the Serre dual of $\m D^\bullet_0$
        \begin{equation*}
            \check{\m D}^\bullet_0: SPEnd_0(\overline{E}^\bullet)^\vee \otimes \O_C(-D)\to  PEnd_0(\overline{E}^\bullet)^\vee \otimes K_C.
        \end{equation*}
        Since $\overline{E}_D^\bullet$ is assumed to have free successive quotients, we have the following duality \cite[Proposition 3.7]{Yokogawa-infinitesimal}
        \begin{equation*}
            SPEnd_0(\overline{E}^\bullet)^\vee \cong PEnd_0(\overline{E}^\bullet)\otimes O_C(D)
        \end{equation*}
        which implies that $\m D_0^\bullet \cong \check{\m D}^\bullet_0$ and hence $\bb{H}^2(\m D^\bullet_0)\cong \bb{H}^0(\m D^\bullet_0)^\vee$. Since $\bb{H}^0(\m D^\bullet_0)= \ker(H^0(PEnd_0(\overline{E}^\bullet))\to H^0(SPEnd_0(\overline{E}^\bullet)))$ and $H^0(PEnd_0(\overline{E}^\bullet))=0$ for an $\a$-stable parabolic Higgs bundle $(\overline{E}^\bullet , \Phi)$. Therefore, we conclude that $\bb{H}^2(\m D^\bullet_0)=0$ and the morphism $\det$ is formally smooth. 
        \end{proof}
            
        \begin{cor}\label{smoothness-of-P}
            The composition $\relativeh \to \relativeho\times _{\nldo} \nld \to \nld $ is smooth. In particular, the fiber $P^{-1}(\uxi)$ is smooth for any $\uxi\in \nld$. 
        \end{cor}
        \begin{proof}
            Note that $\relativeho \to \nldo $ can be written as the composition $\relativeho \xrightarrow{f_1} \Pic^d(C) \times \nldo \xrightarrow{f_2} \nldo$ where $f_1$ is an affine bundle with fibers $H^0(C, K_C(D))$ parametrizing $\Phi$ and $f_2$ is the projection. Both $f_1,f_2$ are clearly smooth morphisms, so $\relativeho\to \nldo$ is smooth as well. Combining with Proposition \ref{smoothness-of-det}, it follows that the composition $\relativeh \to \nld$ is smooth. 
        \end{proof}

\subsection{Spectral correspondence}\label{sec:spectral correspondence}
        As mentioned in the introduction, the key idea to the study of the cohomology of the moduli space of stable $\uxi$-irregular parabolic Higgs bundles is to realize it as a moduli of sheaves on a holomorphic symplectic surface with certain stability condition via the spectral correspondence of \cite{Kontsevich-Soibelman} and \cite{Diaconescu_2018}. 
        
        First we recall the construction of the holomorphic symplectic surface by following \cite{Diaconescu_2018}. By abusing notation, we will also write $M$ as the total space of the twisted canonical line bundle $K_C(D)$. Recall that the polar parts of the Higgs fields are fixed by a choice of $\underline{\xi}= (\xi_1,...,\xi_l)$ with $\xi_i\in H^0(D, M_D)$, which determines a set of divisors $\delta_i\subset M_D$ for $1\leq i\leq l.$ We will call the choice of $\uxi$ \textit{generic} if the leading terms of $\xi_i$ are all distinct. More precisely, if we write $\xi_i= \left(\lambda_{i,n} + \lambda_{i,n-1}z + ...+ \lambda_{i,1}z^{n-1} \right)\frac{dz}{z^n}$ where $\lambda_{i,k}\in \C$, then $\lambda_{i,n}\neq \lambda_{i',n}$ for $i\neq i'$. In this subsection, we will assume that the choice of $\uxi$ is generic. Then we construct a quasi-projective surface $\suxi $ as follows:
        \begin{enumerate}
            \item Let $\f p_{1,i}$ be the intersection point of $\delta_i$ and the reduced fiber $M_p$. The points $\f p_{1,i}$ are distinct points in $M$ under the genericity assumption. First, we simultaneously blow up the surface $M$ at the (reduced) points $\f p_{1,i}$. Denote the resulting surface by $M_1$ and the exceptional divisors by $\Xi_{1,i}$ for each $1\leq i\leq l.$
            \item Suppose $n\geq 2$. For each $i$, let $\f p_{2,i}$ be the intersection point of the strict transform of $\delta_i$ and the exceptional divisor $\Xi_{1,i}$. Then we simultaneously blow up the surface $M_1$ at the (reduced) points $\f p_{2,i}$.  This simultaneous blow-up procedure is then repeated $n-2$ times. The resulting blown-up surface will be denoted by $T_\uxi$ and it contains the exceptional divisors $\Xi_{n,i}$ and the strict transform of the divisors (also denoted by) $\Xi_{a,i}$ for $1\leq a\leq n-1$ and $1\leq i\leq l.$

            \item Let $f$ be the strict transform of the fiber $M_p\subset M$ in $T_\uxi$. Then we obtain the surface $S_\uxi= T_\uxi\setminus (f+\sum_{i=1}^l \sum_{a=1}^{n-1}  \Xi_{a,i}).$
        \end{enumerate}
                Let $\overline{M}=\P(M\oplus \O)$ be the projective completion of $M$. We can apply the same blow-up procedure on $\overline{M}$ and get a projective surface $Z_\uxi$. Then $S_\uxi\subset T_{\uxi}\subset Z_\uxi$. 

        \begin{prop}
            The surface $S_\xi$ is holomorphic symplectic. In other words, It has a trivial canonical line bundle and has a natural compactification by a projective surface $Z_\uxi$. 
        \end{prop}
        \begin{proof}
            The canonical divisor of $T_\xi$ is given by 
            \begin{equation*}
                K_{T_\uxi} = -nf - \sum_{i=1}^l\sum_{a=1}^{n} (n-a)\Xi_{a,i}. 
            \end{equation*}
            Note that the divisors $\Xi_{n,i}$ appear with multiplicity 0 in the expression, so the canonical divisor of the complement $S_\uxi$ of $f+\sum_{i=1}^l \sum_{a=1}^{n-1} \Xi_{a,i}$ in $T_\uxi$ is trivial. 

        \end{proof}

         We will be interested in the moduli space of pure dimension one sheaves supported in $S_\uxi$ with some fixed topological invariants and an appropriate stability condition. To set this up, we first consider the topological invariants. Let $\Delta_i= \sum_{a=1}^n a\Xi_{a,i}\in \Pic(Z_\uxi)$. Let $\sm_0$ be the strict transform of the zero section $C_0$ in $M$. If the support of a pure dimension one sheaf on $Z_\uxi$ is entirely in $S_\uxi$, the class of its support $\ch_1(F)$ must satisfy the following conditions on the intersection numbers:
        \begin{equation*}
            \ch_1(F) \cdot f=0 , \quad \ch_1(F) \cdot \Xi_{a,i}=0 \textrm{  for  }a<n, 1\leq i\leq l
        \end{equation*}
        It is easy to check that such a class must be of the form $\Sigma_{\underline{m}}= r\sm_0 - \sum^l_{i=1} m_i\Delta_i$ where $r\geq 1$ and $\um= (m_1,...,m_l)\in \Z_{\geq 0}^{\times l}$ such that $\sum_{i=1}^l m_i=r$. Note that 
        \begin{equation*}
            \sm_{\underline{m}} \cdot \Xi_{n, i} = m_i\textrm{  for  } 1\leq i\leq l
        \end{equation*}
        
        The stability condition we need will be a slight modification of the usual Gieseker stability condition of coherent sheaves. We review the definition of $\beta$-twisted $A$-Gieseker semistablity condition used in \cite{Matsuki-Wentworth} and \cite[Section 5]{bayermacri}. Let $X$ be a smooth projective surface and $\beta, A\in NS(X)_\Q$ with $A$ ample.  We define the $\beta$-twisted Chern character of $F\in \tx{Coh}(X)$ to be $\ch^\beta(F):=\ch(F) \cup\exp(\beta)\in H^*(X,\Q)$.  Then the
        $\beta$-twisted Hilbert polynomial is defined to be
        \begin{equation*}
            P_\beta(F,t)= \int_X \ch^\beta(F(t))\Td_X.
        \end{equation*}
        Note that when $F$ is a torsion sheaf i.e. $\ch_0(F)=0$, we have 
        \begin{equation*}
            P_\beta(F, t) = \int_X \ch(F(t))\Td_X + \int_X \beta\cdot \ch_1(F(t)) =  P(F,t) + \int_X\beta\cdot \ch_1(F(t)) 
        \end{equation*}
        where $P(F,t)$ denotes the usual Hilbert polynomial. In particular, $P_\b(F, t) = P(F,t)$ when $\b=0$. Let $l(F)/d!$ be the leading coefficient of $P_\b(F,t)$. Define the reduced $\beta$-twisted Hilbert polynomial as $ p_\b(F,t)=P_\b(F,t)/l(F)$. Note that the leading coefficients of $P_\b(F, t)$ and $P(F, t)$ are the same.
        \begin{definition}
            A coherent sheaf $F$ on $X$ is said to be $\beta$-twisted $A$-Gieseker (semi)stable if it has support of pure dimension and for any proper subsheaf $F'\subset F$, one has 
            \begin{equation*}
                p_\b (F',t) < p_\b (F,t)\quad  \textrm{  for  }t\gg 0, \quad (\textrm{resp.  }\leq)
            \end{equation*}
        \end{definition}

        In our case, we will need to make a choice of $\beta$ and an ample line bundle $A$ on $Z_\uxi$. 
        \begin{itemize}
            \item (A choice of $\beta$) For a set of rational numbers $1>\b_1>...>\b_l\geq 0$, we choose $\b= \sum_{i=1}^l \b_i\Xi_{n,i}$ on $NS(Z_\uxi)_\Q$. Then note that $\b \cdot \sm_\um= b_i\Xi_{n,i}\cdot(-m_in\Xi_{n,i}-m_i(n-1)\Xi_{n-1,i} )= -\b_im_i (-n+(n-1))= \b_im_i$. 
            \item (A choice of $A$)  Let $\pi:Z_\uxi\to C$ be the composition of the blow-up morphism $Z_\uxi\to \overline{M}$ and the projection map $q:\overline{M}\to C$. Note that there exists an integer $k$ large enough such that $D_0:=kq^{-1}(p) + C_\infty$ is ample on $\overline{M}$ where $C_\infty$ is the infinity divisor. Then by the result of \cite{Kuchle}, we can always choose $k_1$ large enough such that the divisor $D_1:=k_1\pi_1^*D_0- \sum_{i=1}^l \Xi_{1,i}$ is ample, where $\pi_1:\overline{M_1}\to \overline{M}$ is the initial blow-up map. Proceed inductively, we see that there exists $k_n$ large enough such that $D_n:=k_n\pi_n^*D_{n-1}- \sum_{i=1}^l \Xi_{n,i}$ is ample. Hence, the restriction of the ample divisor $D_n$ to $S_\uxi$ can be written as  $A:=\kappa (\sum_{i=1}^l \Xi_{n,i}) = \pi^* (\kappa p)$ for some large enough $\kappa.$ 
        \end{itemize}
        
        Now we are ready to state the spectral correspondence due to Kontsevich-Soibelman \cite[Section 8.3]{Kontsevich-Soibelman} and Diaconescu-Donagi-Pantev \cite[Section 3.4]{Diaconescu_2018}:
        \begin{theorem}[Spectral correspondence] \label{spectral-corr}
        Fix the numerical data: $r\geq 1, d\in \Z$, $1>\a_1> ...> \a_l\geq 0$ (each $\a_i\in \Q)$, $m_1,...,m_l\in \Z\geq 0$ such that $\sum_{i=1}^l m_i=r$. For a generic choice of $\uxi$, there is an isomorphism between 
            \begin{itemize}
                \item the moduli stack $\frelativehxi$ of  semistable irregular $\uxi$-parabolic Higgs bundles on $C$ of rank $r$, degree $d$, parabolic weights $\underline{\a}=(\a_1,...,\a_l)$, flag type $\underline{m}= (m_1,..,m_l)$;
                \item the moduli stack $\fmsheaves$ of $\beta$-twisted $A$-Gieseker semistable compactly supported pure dimension one sheaves $F$ on $S_\uxi$ with $(\ch_0(F),\ch_1(F), \ch_2(F))= (0, \sm_\um, c)$ where $\sm_\um= r\sm_0 - \sum^l_{i=1}m_i\Delta_i$;
            \end{itemize}
            such that $\a_i=\b_i/n, d=c+r(g-1)$. 
        \end{theorem}
        \begin{proof}
            See Appendix \ref{proof of spectral}.
        \end{proof}

        \begin{rem}
            By the spectral correspondence, for generic $\uxi$, one can also construct the coarse moduli space of $\frelativehxi \cong \fmsheaves$ as an open subset of the coarse moduli space $\muxi$ of the moduli stack of $\b$-twisted $A$-Gieseker semistable sheaves on the surface $Z_\uxi$. With the same numerical data as in Theorem \ref{spectral-corr}, it follows from Theorem \ref{spectral-corr} that the coarse moduli space of stable irregular $\uxi$-parabolic Higgs bundles $\huxi:= P^{-1}(\uxi)$ is isomorphic to an open subset of $\muxi.$
        \end{rem}

\section{Cohomology}
\subsection{Cohomology of the holomorphic symplectic surface}
As a preparation for the next section, we compute (the mixed Hodge structure of) the cohomology of the surface $S_\uxi$. Recall that $\overline{M}=\mathbb{P}(M\oplus \O)$ and $\rho:Z_\uxi \to \overline{M}$ is the iterated blow-up introduced in Section \ref{sec:spectral correspondence}. The surface $S_\uxi$ is defined by the complement of the divisor $Y_\uxi:=C_{\infty}+f+ \sum^l_{i=1} \sum ^{n-1}_{a=1} \Xi_{a,i}$, so $Y_\uxi= Z_\uxi\setminus S_{\uxi}.$ Denote $\iota:S_\uxi \to Z_\uxi$ the canonical inclusion and $q:S_\uxi \to C$ be the projection. By the blow-up formula, the cohomology of $Z_\uxi$ is given by
    \begin{equation*}
       H^k(Z_\uxi,\Q)= \begin{cases}
                    H^2(\overline{M}) \oplus \left(\bigoplus_{1\leq a\leq  m_i,1\leq i\leq l}\Q[\Xi_{a,i}]\right), \quad &k=2 \\
                    H^k(\overline{M},\Q)   = (H^*(C,\Q)\otimes H^*(\P^1,\Q))^k,\quad &k\neq 2
                    \end{cases}
    \end{equation*}

    \begin{prop}\label{prop:coh of surface}
        The cohomology of the surface $S_\uxi$ is given by 
        \begin{equation*}
        H^*(S_\uxi, \Q)\cong H^*(C, \Q) \oplus \left( \bigoplus_{i=1}^l \Q\left[\Xi_{n,i}\right]\right)
        \end{equation*}
        where we write $\left[\Xi_{n,i}\right]:=\iota^*\left[\Xi_{n,i}\right]$ by abusing notation.
              
        \end{prop}
        \begin{proof}
        We denote the complement $Z_\uxi \setminus S_\uxi$ by $Y_\uxi$. By the Mayer-Vietoris sequence, the cohomology of $Y_\uxi$ is given by 
        \begin{equation*}
                H^k(Y_\uxi,\Q) = 
                \begin{cases}
                    \Q, \quad &k=0 \\
                    \Q^{2g},\quad &k=1\\
                    \Q^{l(n-1)+2}, \quad &k=2
                    \end{cases}
            \end{equation*}
            
        Then we have the long exact sequence \begin{center}
        $\begin{array}{c|ccc}
             &H^k_c(S_{\uxi},\Q) &H^k(Z_\uxi,\Q) &H^k(Y_\uxi,\Q)   \\ \hline
            k=0 & 0 &\Q &\Q \\
            k=1 & 0 &\Q^{2g} &\Q^{2g} \\
            k=2 & ? &\Q^{nl+2} &\Q ^{l(n-1)+2}\\
            k=3 & ? &\Q^{2g} &0 \\
            k=4 & \Q &\Q &0 
        \end{array} $
    \end{center}
       By iteratively applying Lefschetz hyperplane theorem, one can see that the morphism $H^k(Z_\uxi,\Q) \to H^k(Y_\uxi,\Q)$ is surjective for $k<2$, hence $\Q\to \Q$ and $\Q^{2g}\to \Q^{2g}$ in the first and second rows are isomorphisms.  Also, clearly $H^4_c(S_\uxi,\Q)=\Q$. To determine $H^k_c(S_\uxi,\Q)$ for $k=2,3$, it suffices to show that the pullback map $H^2(Z_\uxi,\Q) \to H^2(Y_\uxi,\Q)$ is surjective. Note that $H^2(Y_\uxi,\Q)$ can be written as a direct sum of the second cohomology group of each component, which belongs to the image of the pullback map. Therefore we have $H^2_c(S_\uxi, \Q)=\Q^l$ and $H^3_c(S_\uxi,\Q)=\Q^{2g}$. Now the proposition follows from Poincaré duality.

        \end{proof}

Recall that the pure part of the cohomology of a variety $X$ is defined to be $H^k_{\pure}(X, \Q):= \textrm{Gr}^W_kH^k(X,\Q)$ with respect to the mixed Hodge structure on $H^k(X,\Q). $ In particular, when $X$ is smooth, the direct sum of the pure part $H^*(X, \Q)$ is a subalgebra of $H^*(X,\Q)$ since $\textrm{Gr}^W_kH^k(X, \Q)= W_kH^k(X).$ Alternatively, if $i:X\hookrightarrow \overline{X}$ is any smooth compactification, then $H^*_{\pure}(X,\Q)$ can also be defined as the image of the homomorphism $i^*:H^*(\overline{X})\to H^*(X).$ 

Coming back to the previous case, it is easy to see that $H^2(Z_\uxi,\Q)=H^2_{\pure}(Z_\uxi,\Q)$. Moreover, in the proof of Proposition \ref{prop:coh of surface}, the long exact sequence of cohomology groups of the pair $(Z_\uxi, Y_\uxi)$ is compatible with mixed Hodge structures and we showed that the restriction morphism $H^k(Z_\uxi) \to H^k(Y_\uxi)$ is surjective for all $k \geq 0$. This implies that $\textrm{Gr}^W_qH^k_c(S_\uxi,\Q)=0$ except for $q=k$, hence we have $H^*_{c,\pure}(S_\uxi)=H_c^*(S_\uxi)$. By Poincaré duality, we have the following proposition.

\begin{prop}\label{prop: purity of suxi}
    The cohomology of the surface $S_\uxi$ is of pure type. In other words, $H^*_\pure(S_\uxi,\Q)=H^*(S_\uxi,\Q)$.
\end{prop}

\subsection{Generators}
In this section, we study the cohomology of the moduli space $\huxi:= P^{-1}(\uxi)$ for a generic $\uxi.$ Assume that $r, d$ are coprime such that there exist a universal  bundles $\euxi$ on $\huxi\times C$  and a universal flag of subbundles $\euxiflagd$ on $\huxi\times D$.  Let $\euxiflagp$ be the restriction of $\m E^\bullet_{\huxi\times D}$ to $\huxi\times \{p\}= \huxi$, then we define the vector bundle $\quxi^i = \m E^i_{\huxi\times p}/\m E^{i-1}_{\huxi\times p}$ on $\huxi$.

Let $\m M_\uxi$ to be the moduli space of $\b$-twisted $A$-Gieseker stable sheaves on $Z_\uxi$ with fixed Chern characters $(\ch_0(F),\ch_1(F), \ch_2(F))= (0, \sm_\um, c)$ where $\sm_\um= r\sm_0 - \sum^l_{i=1}m_i\Delta_i$. Assume that $\b$ and $A$ are chosen generic enough such that every $\b$-twisted $A$-Gieseker semistable sheaf is $\b$-twisted $A$-Gieseker stable. In particular, $\m M_\uxi$ is projective. Moreover, the moduli space $\m H_\uxi$ is an open subset in $\m M_{\uxi}$ via the spectral correspondence. Since the ample divisor $A$ is chosen such that $A$ restricted to $S_\uxi$ is $\pi^*(\kappa p)$ for some sufficiently large $\kappa$, we have $P(F, t) = P(\pi_*F, \kappa t)= r\kappa t + d+r(1-g)$ for each $F\in \muxi$. By applying \cite[Corollary 4.6.6]{huybrechts-lehn}, the assumption of $r$ and $d$ being coprime implies that there exists a universal sheaf $\m F_{\m M_{\uxi}}$ on $\m M_{\uxi}\times Z_\uxi$. We will denote by $\m F$ its restriction to $\m H_\uxi\times Z_\uxi$.

\begin{prop}\label{prop:pure generator}
    The K\"{u}nneth components of the Chern classes of $\m F$ generate $H_{\pure}^*(\m H_\uxi,\Q).$
\end{prop}
\begin{proof}
    Let $H^*_{\ch(\m F)}(\huxi, \Q)\subset H^*(\huxi, \Q)$ be the subring generated by the K\"{u}nneth component of the Chern classes of $\m F$. Then it is clear that $H^*_{\ch(\m F)}(\huxi, \Q)\subset H^*_{\pure}(\huxi, \Q)$. Recall that each $F\in \m H_{\uxi}$ viewed as a $\b$-twisted $A$-Gieseker stable sheaf via the spectral correspondence is supported on the open surface $S_{\uxi}\subset Z_{\uxi}$ and $S_{\uxi}$ is holomorphic symplectic. In particular, we have $F\otimes K_{Z_{\uxi}} \cong F\otimes K_{S_{\uxi}} \cong F.$ Then this is the setup where we can apply the argument and result of Markman \cite[Section 4]{markman2002symplectic}. 
    
    If $G$ is another $\b$-twisted $A$-Gieseker stable sheaf in $\m M_{\uxi}$, we have 
    \begin{align}\label{ext2}
        \Ext^2_{Z_{\uxi}}(F, G)&\cong \Hom (G, F\otimes K_{Z_{\uxi}})^\vee\cong  \Hom(G, F)^\vee ; \\ \label{ext22}
        \Ext ^2_{Z_{\uxi}}(G, F) &\cong \Hom (F, G\otimes K_{Z_{\uxi}} )^\vee\cong  \Hom (F\otimes K_{Z_{\uxi}}^{-1}, G)^\vee\cong \Hom (F, G)^\vee.
    \end{align}
    Let us denote by $\pi_{ij}$ the projection from $\m M_{\uxi}\times Z_{\uxi}\times \m H_{\uxi}$ to the product of the $i$-th and $j$-th factors. For any flat projective morphism $f:X\to Y$ and coherent sheaves $F_1,F_2$ on $X$, we denote by $\m Ext^i_f(F_1, F_2) := R^if_*\m Hom(F_1, F_2)$ the $i$-th relative extension sheaf on $Y$ and $\m Ext^!_f(F_1, F_2):= \sum (-1)^i\m Ext^i_f(F_1, F_2)$ the corresponding class in the Grothendieck group of $Y$. Then the identities \eqref{ext2}  and \eqref{ext22} above imply that the following relative extension sheaves 
    \begin{align*}
        \m Ext^2_{\pi_{13}}(\pi_{12}^* \m F_{\m M_{\uxi}} , \pi_{23}^*\m F) , \quad \m Ext^2_{\pi_{13}}(\pi_{23}^* \m F , \pi_{12}^*\m F_{\m M_{\uxi}}) 
    \end{align*}
    are supported as line bundles on the graph $\Delta \subset \m M_{\uxi}\times \m H_{\uxi}$ of $\huxi\hookrightarrow \muxi$.  Then we can apply the proof of \cite[Theorem 1]{markman2002symplectic} verbatim to $\Delta\subset \muxi\times \huxi$ to show that the class of $\Delta$ in the Borel-Moore homology is given by the image of the Poincar\'{e}-duality map of
    \begin{equation}\label{markman's equality}
         c_m [-\m Ext^!_{\pi_{13}}(\pi^*_{12}\m F_{\muxi}, \pi^*_{23}\m F)].
    \end{equation}
    As remarked in \cite[Section 4]{markman2002symplectic}, the argument for the equality \eqref{markman's equality} does not require the smoothness of $\muxi$. 

    Let $f: \td{\m M_{\uxi}}\to \muxi$ be a resolution of $\muxi$ which is a then a smooth compactification of $\huxi$. Denote the inclusion by $i:\huxi\hookrightarrow \td{\muxi}$ and its graph by $\td{\Delta} \subset \td{\muxi}\times \huxi$. Let $H^*_{\td{\Delta}}(\huxi, \Q)\subset H^*(\huxi,\Q)$ be the linear subspace spanned by the right hand K\"{u}nneth components of the class of $\td{\Delta}$. Since $H^*_{\pure}(\huxi, \Q)= Im(H^*(\td{\muxi},\Q)\to H^*(\huxi,\Q)) 
    $, a standard argument \cite[Proposition 2.1]{mcgerty-nevins} shows that $H_{\pure}^* (\huxi, \Q)\subset  H_{\td{\Delta}}^*(\huxi, \Q)$. Finally, since $\Delta$ is pulled back to $\td{\Delta}$ under the natural map $f\times \Id:\td{\muxi}\times \huxi \to \muxi\times \huxi$, the class of $\td{\Delta}$ is given by $c_m [-\m Ext^!_{\pi_{13}}((f\times \Id\times \Id)^*\pi^*_{12}\m F_{\muxi}, \pi^*_{23}\m F)].$ It follows that $H^*_{\td{\Delta}}(\huxi, \Q) \subset H^*_{\ch(\m F)}(\huxi, \Q)$ (see \cite[Corollary 2]{markman2002symplectic}). Combining the inclusions of subrings, we conclude that $H^*_{\ch(\m F)}(\huxi,\Q)= H^*_{\pure}(\huxi,\Q)$. 
\end{proof}

    Define the linear map 
    \begin{equation*}
        \ch_{\m F}(-) : H^*(\zuxi, \Q)\to H^*(\huxi, \Q), \quad \a \mapsto  q_{\huxi, *} \{\ch(\m F)\cdot q^*_{\zuxi} (\a\cdot \Td_{\zuxi}) \}.
    \end{equation*}
    Note that the K\"{u}nneth components of the Chern classes of $\m F$ in $H^k(\huxi,\Q)$ can always be expressed as the projection of the image of $\ch_{\m F}(-)$  from $H^*(\huxi,\Q)$ to $H^k(\huxi, \Q)$. If $\a\in K=\ker(H^*(\zuxi,\Q)\to H^*(\suxi,\Q))$, then $\ch(\m F) \cdot q^*_{\zuxi}(\a)$ vanishes since the universal sheaf $\m F$ is supported on the open subset $\m H_{\uxi}\times S_\uxi\subset \m H_{\uxi}\times Z_\uxi$. It follows that the map $\ch_{\m F}(-)$ factors through $H^*(\zuxi,\Q)/K\cong H^*_{\pure}(\suxi)  $ i.e. the pure part of $H^*(\suxi, \Q)$, which is $H^*(\suxi,\Q)$ itself due to Proposition \ref{prop: purity of suxi}. So we can also write $\ch_{\m F}(-): H^*(\suxi, \Q)\to H^*(\huxi,\Q)$. Since $H^*(\suxi, \Q) \cong H^*(C,\Q)\oplus \bigoplus_{i=1}^l\Q[\Xi_{n,i}]$ by Proposition \ref{prop:coh of surface}, we are going to describe the image of each component under $\ch_{\m F}(-)$. 

     Recall the following notations: $\pi: Z_{\uxi}\to C$ is the composition of the blow-up morphism and the projection $\zuxi\to \overline{M}\to C$. Let $\nu_i: \Xi_{n,i}\to \zuxi$ be the closed embedding of the exceptional divisor for $1\leq i\leq l.$ Denote by $q_{\huxi},q_{\zuxi}$ the projection from $\huxi\times \zuxi$ to the corresponding factors and $p_{\huxi},p_{C}$ the projection from $\huxi\times C$ to the corresponding factors. The notation is summarized in the following diagram.
    \begin{equation*}
        \begin{tikzcd}
            \huxi \arrow[d, equal] & \huxi\times \zuxi \arrow[l, swap, "q_{\huxi}"] \arrow[d, "\Id \times \pi"] \arrow[r, "q_{\zuxi}"] & \zuxi \arrow[d, "\pi"] \\
            \huxi & \huxi \times C \arrow[l, swap, "p_{\huxi}"] \arrow[r, "p_C"] & C
        \end{tikzcd}
    \end{equation*}
    
    \begin{lem}
    Every element in the image $\ch_{\m F}(H^*(C, \Q))\subset H^*(\huxi, \Q)$ can be written as a linear combination of the K\"{u}nneth components of the Chern classes of $\euxi$.
    \end{lem}
    \begin{proof}
        Note that by the spectral correspondence, $(\Id \times \pi)_*\m F = \euxi$. 
        Suppose $\pi^*(x)\in H^*(\suxi,\Q)$ is a class from $H^*(C,\Q)$. Using $q_{\huxi} = p_{\huxi}\circ (\Id \times \pi) $ and $\pi\circ q_{\zuxi} = p_C\circ (\Id \times \pi)$, we have 
        \begin{align*}
           \ch_{\m F}(\pi^*(x))&=  q_{\huxi,*}\{\ch(\m F) \cdot q^*_{\zuxi}(\pi^*(x) \cdot \Td_{\zuxi})\} \\
           &= p_{\huxi, *}\circ (\Id\times \pi)_*  \{ \ch(\m F ) \cdot q^*_{\zuxi}  \Td_{\zuxi} \cdot (\Id\times \pi)^* p_C^* (x)\}\\
            &= p_{\huxi, *} \{ (\Id\times \pi)_*(\ch(\m F) \cdot q^*_{\zuxi}  \Td_{\zuxi}) \cdot p^*_C(x)\}\\
            &=p_{\huxi, *}\{ \ch(\euxi) \cdot p_C^*\Td_C \cdot p^*_C(x)\}
        \end{align*}
        where we used the projection formula for the second equality and the Grothendieck-Riemann-Roch theorem for the third equality. 
        
    \end{proof}

    \begin{lem}
        The image $\ch_{\m F} ([\Xi_{n,i}])$ in $H^*(\huxi, \Q)$ is equal to $\ch(\quxi^i)$. 
    \end{lem}
    \begin{proof}
        Denote by $\iota: \m H_\uxi\times \{p\}\to \m H_\uxi\times C$ the inclusion and $p'_{\m H_\uxi}: \m H_\uxi\times \{p\}\to \m H_\uxi$ the projection.  Note that 
    \begin{equation*}
        \ch(\quxi^i)= p'_{\huxi *}\left\{(p'_{\huxi})^*\ch( \quxi^i)\right\}=p_{\huxi,*}\iota_* \{(p'_{\huxi})^*\ch( \quxi^i) \} = p_{\huxi*}\left\{\ch \left(\iota_* \quxi^i\right)\cdot p^*_C \Td_C \right\}
    \end{equation*}    
    where we apply the Grothendieck-Riemann-Roch theorem to the last equality. 

    By the construction of the spectral correspondence and Remark \ref{associated graded}, we have 
    \begin{equation*}
        \iota_*\quxi^i  \cong (\Id\times \pi)_* (\m F\otimes (\Id\times \nu_i)_*\O_{\huxi\times \Xi_{n,i}})
        \end{equation*}
     Then 
     \begin{align*}
        p_{\huxi*}\left\{\ch \left(\iota_* \quxi^i\right)\cdot p^*_C \Td_C \right\} &= p_{\huxi*}\left\{\ch \left((\Id\times \pi)_* (\m F\otimes (\Id\times \nu_i)_*\O_{\huxi\times \Xi_{n,i}})\right)\cdot p^*_C \Td_C \right\} \\
        &= p_{\huxi*}\left\{(\Id\times \pi)_*\left( \ch (\m F\otimes (\Id\times \nu_i)_*\O_{\huxi\times \Xi_{n,i}})\cdot q^*_{\zuxi} \Td_{\zuxi}\right)\right\} \\
        &=q_{\huxi*} \{ \ch(\m F) \cdot [\huxi\times \Xi_{n,i}] \cdot q^*_{\zuxi} \Td_{\zuxi} \} \\
        &= q_{\huxi*} \{ \ch(\m F) \cdot q_{\zuxi}^* [ \Xi_{n,i}] \cdot q^*_{\zuxi} \Td_{\zuxi} \} = \ch_{\m F}([\Xi_{n,i}]).
     \end{align*}
     where the second equality follows from the Grothendieck-Riemann-Roch theorem and the third inequality we use the fact that $\ch(f_*\O_Y) = [Y]$ for a Cartier divisor $f:Y\hookrightarrow X $. 
        Therefore, the expressions above combine to give $\ch(\quxi^i) = \ch_{\m F} ( [\Xi_{n,i}])$.
    \end{proof}

    Combining Proposition \ref{prop:pure generator} with the two previous lemmas, we obtain our main theorem. 
    \begin{theorem}\label{thm:purgen}
        Let $\uxi$ be generic. The pure cohomology $H^*_{pure}(\huxi, \Q)$ is generated by the K\"{u}nneth components of the Chern classes of $\euxi$ and the Chern classes of $\quxi^i$, where $1\leq i\leq l.$ 
    \end{theorem}

\subsection{Purity}\label{sec:purity}

In this section, we study the purity of the $\Q$-mixed Hodge structure on $H^*(\huxi,\Q)$ in the regular full-flag case. For simplicity, we write $\mathcal{H}={\bf \mathcal{H}}(C,p;r,d,\underline{\a}, \underline{1})$, $\mathcal{N}=\mathcal{N}(l,p)$, $P:\m H\to \m N$, $\mathcal{H}_\uxi=P^{-1}(\uxi)$, $\m M= \mathcal{M}(C, p;r,d,\ua, \underline{1})$ and $\m M_0= \mathcal{M}_0(C, p;r,d,\ua, \underline{1})$. Recall that by Proposition \ref{prop:isomodulispace}, we have $\m H\cong \m M$. Also, by definition, we have $\m M_0=P^{-1}(0)$. 

\begin{theorem}\label{thm: main full flag}

In the regular full-flag case, for any $\uxi \in \mathcal{N}$, the restriction morphism 
    \[
    \iota_\uxi^*:H^*(\mathcal{H},\Q) \to H^*(\huxi,\Q)
    \]
    is an isomorphism of $\Q$-mixed Hodge structures where $\iota_\uxi:\huxi \hookrightarrow \mathcal{H}$ is the canonical inclusion. In particular, both are of pure type. 
\end{theorem}

The proof relies on the property of semi-projective varieties.

\begin{definition}\cite[Definition 1.1.1]{HauselRodriguez2013}
    Let $X$ be a non-singular quasi-projective variety over $\C$ with a $\C^\ast$-action. We call $X$ semi-projective if the following two conditions hold:
    \begin{enumerate}
        \item the fixed point set $X^{\C^\ast}$ is proper
        \item for every $x \in X$, the limit $\lim_{\lambda \to 0}\lambda x$ exists for $\lambda \in \C^\ast$.
    \end{enumerate}
\end{definition}

\begin{prop}\cite[Corollary 1.3.3]{HauselRodriguez2013}\label{prop:semiprojHausel}
    Let $X$ be a non-singular complex algebraic variety and $f:X \to \C$ is a smooth morphism. Also, suppose that $X$ is semi-projective with a $\C^\ast$-action making $f$ equivariant covering a linear action of $\C^\ast$ on $\C$ with positive weight. Then the fibers $X_c:=f^{-1}(c)$ have isomorphic cohomology supporting pure mixed Hodge structures. 
\end{prop}

\begin{lem}
    In the regular full-flag case, the relative moduli space $\mathcal{H}$ is semi-projective.
\end{lem}
\begin{proof}
First, note that we have a natural scaling $\C^\ast$-action on $\mathcal{H}$: for $t \in \C^\ast$, $(E, E_D^\bullet, \alpha, \Phi) \in \mathcal{H}$, the action is defined as $t\cdot (E, E_D^\bullet, \alpha, \Phi):=(E, E_D^\bullet, \alpha, t\Phi)$. Similarly, there is a scaling $\C^\ast$-action on $\m M$ such that the isomorphism $\mathcal{H} \cong \m M$ we showed in Proposition \ref{prop:isomodulispace} is compatible with these scaling $\C^\ast$-actions. The conclusion follows from  the properness of the Hitchin morphism $h$ \cite[Corollary 5.12] {Yokogawa-compactification} which shows that $\m M$ is semi-projective with respect to the scaling $\C^\ast$-action. 
\end{proof}

\begin{proof}[Proof of Theorem \ref{thm: main full flag}]
By iteratively applying Proposition \ref{prop:semiprojHausel} to the canonical morphism $P:\mathcal{H} \to \mathcal{N}$, we achieve the purity of the $\Q$-mixed Hodge structure on $H^*(\huxi)$ for any $\uxi\in \mathcal{N}$. Since $P$ is smooth by Corollary \ref{smoothness-of-P}, the pullback morphism 
\[
\iota_\uxi^*:H^*(\mathcal{H}) \to H^*(\huxi)
\]
induced by the canonical inclusion $\iota_\uxi:\huxi \hookrightarrow\mathcal{H}$ is an isomorphism of $\Q$-mixed Hodge structures by the proof of \cite[Corollary 1.3.3]{HauselRodriguez2013}. This completes the proof. 
\end{proof}

\begin{rem}
    In the work of Komyo \cite{komyo} and \cite{hausel2022pw},a similar argument was used to show that the cohomology of the coarse moduli space of regular parabolic Higgs bundles is of pure type. 
\end{rem}

Following from Theorem \ref{thm:purgen}, we have the following generation result for the cohomology ring of the coarse moduli space of regular (strongly) parabolic Higgs bundles of the full-flag type. 

\begin{cor}\label{cor:generator reg par}
    The cohomology of the coarse moduli space of stable regular parabolic (resp. strongly parabolic) Higgs bundles $\mathcal{M}$ (resp. $\mathcal{M}_0$) of the full-flag type is generated by the Künneth components of the Chern classes of a universal bundle  and the Chern classes of successive quotients of a universal flag of subbundles. 
\end{cor}
\begin{proof}
    Recall that by Proposition \ref{prop:isomodulispace}, we have $\m H\cong \m M$. Choose a generic $\uxi$ and denote by $\iota'_{\uxi}:\huxi\to \m H\xrightarrow{\sim }\m M$ the composition. Let $\m E$ be a universal bundle on $\m M\times C$ and $\m E_{\uxi}$ be a universal bundle on $\huxi\times C$ such that $(\iota'_\uxi\times \Id)^*\m E = \m E_\uxi$. Let $\m E^\bullet_{\m M\times p}$ (resp. $\m E^\bullet_{\huxi\times p}$) be the universal flag of subbundles of $\m E$ (resp. $\m E_\uxi$)  on $\m M\times p$ (resp. $\huxi\times p$). Let $Q^i:=\m E^i_{\m M\times p}/\m E^{i-1}_{\m M\times p}$ and $Q^i_\uxi:=\m E^i_{\huxi\times p}/\m E^{i-1}_{\huxi\times p}$ be their associated quotients respectively. Clearly, $(\iota'_\uxi)^*\m E^\bullet_{\m M\times p} = \m E^\bullet_{\huxi\times p}$ and $Q^i_\uxi$ is isomorphic to $(\iota')_\uxi^*Q^i$. It is easy to see that the left-hand K\"{u}nneth components of the Chern classes of $\m E_\uxi$ in $H^*(\huxi)$ are exactly the pullback of the left-hand K\"{u}nneth components of the Chern classes of $\m E$ in $H^*(\m H)$ along $\iota'_\uxi\times \Id$. Moreover, we have $(\iota'_\uxi)^*\ch(Q^i) = \ch((\iota'_\uxi)^*Q^i)= \ch(Q^i_\uxi).$ Since $(\iota'_\uxi)^*: H^*(\m M)\to H^*(\huxi)$ is an isomorphism by Theorem \ref{thm: main full flag}, it follows from the generation result of $H^*(\huxi)$ (Theorem \ref{thm:purgen}) that $H^*(\m M)$ is also generated by the K\"{u}nneth components of the Chern classes of $\m E$ and the Chern classes of $Q^i$, where $1\leq i\leq l.$

    For the case of the coarse moduli space $\m M_0$ of stable regular strongly parabolic Higgs bundles of the full-flag type, note that $\m M_0= P^{-1}(0)$.  So Theorem \ref{thm: main full flag} again implies that $ H^*(\m H)\cong H^*(\m M_0)$. Applying the same argument as above shows the desired result. 
\end{proof}

We conclude this section by explaining why the same method is not applicable to study the purity when $n>1$. The issue is that $\relativeh$ is not neccessarily semi-projective because the locus of fixed point $\relativeh^{\C^\ast}$ may not be proper: consider the morphism $G:\relativeh \to \mathcal{M}(C, D;r, d, \underline{\a}, \underline{m})$, introduced in Remark \ref{rem:comparison fails}. It is easy to see that $G$ is $\C^\ast$-equivariant. Also, since the scaling $\C^\ast$-action is equivariant with respect to the Hitchin morphism $h:\mathcal{M}(C, D;r, d, \underline{\a}, \underline{m}) \to B$, the fixed locus belongs to $P^{-1}(0) \cap h^{-1}(0)$. While $h^{-1}(0)$ is proper, the intersection $P^{-1}(0) \cap h^{-1}(0)$ is in general open. This is because the freeness condition on the successive quotients of the flag $E_D^\bullet$ is an open condition. For example, take $D=3p$, $\underline{m}=(3,3)$ and $E$ is a trivial bundle of rank $2$. The restriction $E_D$ is a free $k[t]/t^3$-module of rank 2, $k[t]/t^3\oplus k[t]/t^3$. Let $\Phi$ be a Higgs field whose local form at $D$ is
\[
\Phi|_D=\begin{bmatrix}
    0 & t \\
    0 & 0
\end{bmatrix}
\]
Choose a filtration $0 \subset E_D^1 \subset E_D$ where $E_D^1=(t^2)\oplus (t)$. This filtration is preserved by $\Phi$ while the quotient $E_D/E_D^1$ is not free. 

On the other hand, for a generic $\uxi$, it can be shown that that the freeness of successive quotients $E_D^i/E_D^{i-1}$ follows automatically from the $\uxi$-parabolic condition.  This implies that $\huxi$ is a closed subscheme in $\mathcal{M}(C, D;r, d, \underline{\a}, \underline{m})$. 

\begin{conj}
    When $n>1$, the $\Q$-mixed Hodge structure on $H^*(\mathcal{H}_\uxi)$ is of pure type for any generic $\uxi$.
\end{conj}

\appendix
\section{Proof of spectral correspondence} \label{proof of spectral}
    In this appendix, we provide a proof of the spectral correspondence (Theorem \ref{spectral-corr}) as presented in \cite{Diaconescu_2018} (in one direction) with a slight modification on the stability conditions. We will follow the notations in Section 2.2.

     Let $F\in \fmsheaves$ be a pure dimension one sheaf on $S_\uxi$. Then we get a Higgs bundle $E=\pi_* F$ on $C$ with Higgs field $\Phi:E\to E\otimes K_C(D)$ obtained from the direct image of the multiplication map $F\to F\otimes \O_{S_\uxi}(\sm_0)$. The filtration on $E_D$ is obtained as follows. First, observe that there are the surjections $p_{i-1}: F\twoheadrightarrow F\otimes \O_{n\Xi_{n,i}}$ for $1\leq i\leq l.$ We can then define $F_{l-1}:= \ker(p_{l-1}).$ Since the divisors $\Xi_{n,i}$ are mutually disjoint, the composition $F_{l-1}\xrightarrow{\iota_{l-1}} F\xrightarrow{p_{l-2}} F\otimes \O_{n\Xi_{n,l-2}}$ is also surjective and we define $F_{l-2}$ to be $\ker(p_{l-2}\circ \iota_{l-1})$ which is contained in $F_{l-1}$. By constrction, we have $F_{l-1}/F_{l-2}= F\otimes \O_{n\Xi_{n,l-1}}$. Repeating this construction, we obtain a filtration of sheaves 
    \begin{equation*}\label{filtration-on-F}
                F^\bullet: F_0\subset F_1\subset... \subset F_{l-1}\subset F
            \end{equation*}
            such that $F_i/F_{i-1}\cong F\otimes \O_{n\Xi_{n,i}}$.

            Since $\pi$ restricted to $\suxi$ is affine, the functor $(\pi|_{\suxi})_*$ is exact. Since the sheaf $F$ is supported in $\suxi$, applying the functor $\pi_*$ and tensoring with $\O_D$ to the surjection $F\to F \otimes \O_{n\Xi_{n,l}}$, we obtain the surjection 
            \begin{equation*}
                q_{l-1}:E_D= (\pi_*F)_D\twoheadrightarrow \pi_*(F\otimes \O_{n\Xi_{n,l}})\otimes \O_D
            \end{equation*}
            Note that $\pi_*(F\otimes \O_{n\Xi_{n,l}})\otimes \O_D\cong \pi_*(F\otimes \O_{n\Xi_{n,l}}) $ since $\pi_*(F\otimes \O_{n\Xi_{n,l}})= \mu_*(\pi_*\nu_i^* F)$ where $\mu:D\to C$ and $\nu_i:n\Xi_{n,i}\to \zuxi$ denote the closed embeddings. We define  $E^{l-1}_D:=\ker(q_{l-1})\subset E_D$ as the first step in the filtration. The second step of the filtration is induced from the surjection $F_{l-1}\xrightarrow{\iota_{l-1}} F\xrightarrow{p_{l-2}} F\otimes \O_{n\Xi_{n,l-1}}$. We apply $\pi_*$ followed by tensoring with $\O_D$ to $p_{l-2}\circ \iota_{l-1}$ as before and obtain the surjection 
            \begin{equation*}
              (E_{l-1})_D \to E_D \to \pi_*(F\otimes \O_{n\Xi_{n,l-1}}) 
            \end{equation*}
            where $E_{l-1}:=\pi_*(F_{l-1})$. Since the composition $(E_{l-1})_D\to E_D\xrightarrow{q_{l-1}} \pi_*(F\otimes \O_{n,\Xi_{n,l}})$ is zero, the map $(E_{l-1})_D\to E_D$ factors through $E^{l-1}_D$. Then we define $E^{l-2}_D := \ker (E^{l-1}_D\to E_D\to  \pi_*(F\otimes \O_{n\Xi_{n,l-1}}))$.
            Proceed iteratively, we get a filtration
            \begin{equation*}
                E_D^\bullet:     0\subset E^1_D\subset...\subset E^{l-1}_D\subset E_D. 
            \end{equation*}
            Note that by construction $ E^i_D/E^{i-1}_D\cong \pi_*(F\otimes \O_{n\Xi_{n,i}}).$  Since the Higgs field is induced from the multiplication map $F\to F\otimes _{S_\uxi}\O_{\uxi}(\sm_0)$ which preserves the filtration $F^\bullet$, it follows that the Higgs field also preserves $E_D^\bullet$.

            To see that $(E,E_D^\bullet, \Phi)$ is $\uxi$-parabolic, we write the Higgs field in local coordinate. Choose an affine coordinate chart $(U, u,z)$ on $M$ with horizantal coordinate $z$ and vertical coordinate $u$ centered at the intersection between the zero section and $M_p$. Let $y\in H^0(M,\rho^*M)$ be the tautological section. Then we can write   
    \begin{equation*}
        y|_U= u\frac{dz}{z^n}
    \end{equation*}
    where $dz/z^n$ is viewed as a local frame. Fix the integer $i$. It suffices to check the condition \eqref{xi-parabolic} around the point $\f p_{1,i}$ and restrict the simultaneous blow-up at $l$ points in each step to blowing up at a single point each step.  We denote the $n$ blow-ups by $U_n\to U_{n-1} \to ... \to U_1\to U$ and $(u_j,z_j)$, $j=1,...,n$, the affine coordinate for $U_j$. Set $u_0= u , z_0=z$.
    
    Each $\xi_i$ can be represented as 
    \begin{equation*}
        \xi_i= \left(\lambda_{i,n} + \lambda_{i,n-1}z + ...+ \lambda_{i,1}z^{n-1} \right)\frac{dz}{z^n}
    \end{equation*}
    where $\lambda_{i,k}\in \C$. Then the first blow-up is given in an affine chart as 
    \begin{equation*}
        u_0 - \lambda_{i,n} = u_1z_0, \quad z_1=z_0
    \end{equation*}
    Similarly, for each $1\leq j\leq n$, we have 
    \begin{equation*}
        u_j - \lambda_{i,n-j} = u_{j+1}z_j, \quad  z_j= z_{j+1}
    \end{equation*}
    Combining the equations, we get that 
    \begin{equation*}
        u= u_1z_1+\lambda_{i,n}  = (u_2z_2+ \lambda_{i,n-1})z_2+ \lambda_{i,n} = ...= u_nz_n^n+ \lambda_{i,1}z^{n-1}_n+ ... + \lambda_{i,n-1}z_n+ \lambda_{i,n}, \quad z=z_n
    \end{equation*}
    The equations $u= u_nz_n^n+ \lambda_{i,1}z^{n-1}_n+ ... + \lambda_{i,n-1}z_n+ \lambda_{i,n}$ and $z=z_n$ define the morphism between the affine neighbourhoods of $U_n$ and $U$. Note that 
    \begin{equation*}
        u\frac{dz}{z^n} = u_ndz + \left(\lambda_{i,1}\frac{dz}{z}+ ... + \lambda_{i,n-1}\frac{dz}{z^{n-1}}+ \lambda_{i,n}\frac{dz}{z^n}\right)
    \end{equation*}
    The space of sections of $F$ in $U_n$ is a $\C[u_n,z]$-module $\Upsilon_F$. By \cite[Appendix A.3]{Diaconescu_2018}, $\pi_*(F\otimes \O_{n\Xi_{n,i}})$ is a locally free $\O_D$-module, which implies that its space of sections $\Upsilon_F \otimes_{\C[u_n,z]}\C[u_n,z]/(z^n)\cong B^{\oplus m_i}$ where $B=\C[z]/(z^n)$. Then the Higgs field restricted to $D$ is a $B$-module homomorphism $\Phi_{D,i}:B^{\oplus m_i}\to B^{\oplus m_i}\otimes dz/z^n$ given by multiplication by $u\frac{dz}{z^n}$. But
    \begin{equation*}
        u=\lambda_{i,1}z^{n-1}_n+ ... + \lambda_{i,n-1}z_n+ \lambda_{i,n}  \qquad \textrm{mod }B
    \end{equation*}
    Therefore, multiplication by $udz/z^n$ is the same as multiplication by $\xi_i\otimes \Id_{B^{\oplus m_i}}$ and $\Phi_{D,i} = \xi_i\otimes \Id_{E^i_D/E^{i-1}_D}$. This concludes the assignment of an irregular $\uxi$-parabolic Higgs bundle for each $F\in \fmsheaves$. For the other direction, we refer the readers to \cite[Section 3.3]{Diaconescu_2018}. 
    
    The equivalence of the stability conditions of both sides can be seen as follows. Recall that $\beta= \sum^l_{i=1} \b_i\Xi_{n,i} $ for some choice of $\b_i, 1\leq i\leq l$. Then the difference between $P_\b(F, t)$ and $P(F, t)$ becomes
        \begin{equation*}
            \int_X \beta\cdot \ch_1(F(t)) = \int_X \beta\cdot \ch_1(F)  = \int \beta\cdot \sm_\um = \sum \beta_im_i
        \end{equation*}
        Moreover, note that $nm_i= \chi(F\otimes \O_{n\Xi_{n,i}}) = \chi(\pi_*(F\otimes \O_{n\Xi_{n,i}})) = \chi (E^i_D/E^{i-1}_D)$. If $\beta$ is chosen such that $\a_i=\b_i/n$, then 
        \begin{equation}\label{difference-term}
            \int_X \beta\cdot \ch_1(F(t)) = \sum_{i=1}^l \a_i\chi(E^i_D/E^{i-1}_D)
        \end{equation}
        Recall that the ample divisor $A$ is chosen such that $A$ restricted to $S_\uxi$ is $\pi^*(\kappa p)$ for some sufficiently large $\kappa$, so we have $P(F, t) = P(\pi_*F, \kappa t)$. Therefore, combining with the equality \eqref{difference-term} above combine yields
        \begin{equation}\label{equiofpolynomials}
            p_\beta(F,t) = \para P_\a(\pi_* F, \kappa t).
        \end{equation}
        Now, the proof in \cite[Section 3]{Diaconescu_2018} also shows that under the correspondence, proper coherent subsheaves $F'$ of $F$  of pure dimension one on $S_\uxi$ correspond to proper coherent subsheaves $E'=\pi_*F'$ of $E=\pi_*F$ on $C$ preserved by $\Phi$ with the induced filtration $F'^\bullet_D.$ Since scaling the variable in a polynomial preserves the order, the equality \eqref{equiofpolynomials} implies the equivalence of the $\beta$-twisted $A$-Gieseker (semi)stability and the stability of irregular parabolic Higgs bundles.
        
        \begin{rem}\label{associated graded}
            Note that if we replace $\O_{n\Xi_{n,i}}$ in the proof above by $\O_{\Xi_{n,i}}$, then the construction of the filtration yields a filtration of vector spaces $E^\bullet_p: 0\subset E^1_p\subset\dots \subset E^{l-1}_p\subset E_p$ such that  $E^i_p/E^{i-1}_p\cong \pi_* (F\otimes \O_{\Xi_{n,i}}).$ Also, it is clear that the restriction of $E^\bullet_D$ to $p$ is $E^\bullet_p$.  
        \end{rem}

\section*{Acknowledgement}
We would like to thank Philip Boalch, Michi-aki Inaba, Tony Pantev, Qizheng Yin for many useful discussions. The second author gratefully acknowledges the financial support received from the Leverhulme Trust.

\printbibliography

@article{Diaconescu_2018,
	doi = {10.1007/s00220-018-3097-9},
  
	url = {https://doi.org/10.1007%2Fs00220-018-3097-9},
  
	year = 2018,
	month = {02},
  
	publisher = {Springer Science and Business Media {LLC}
},
  
	volume = {359},
  
	number = {3},
  
	pages = {1027--1078},
  
	author = {Duiliu-Emanuel Diaconescu and Ron Donagi and Tony Pantev},
  
	title = {{BPS} States, Torus Links and Wild Character Varieties},
  
	journal = {Communications in Mathematical Physics}
}

@article {markman2006integral,
    AUTHOR = {Markman, Eyal},
     TITLE = {Integral generators for the cohomology ring of moduli spaces
              of sheaves over {P}oisson surfaces},
   JOURNAL = {Adv. Math.},
  FJOURNAL = {Advances in Mathematics},
    VOLUME = {208},
      YEAR = {2007},
    NUMBER = {2},
     PAGES = {622--646},
      ISSN = {0001-8708,1090-2082},
   MRCLASS = {14D20 (14C15 14F05 14J60)},
  MRNUMBER = {2304330},
MRREVIEWER = {Zhenbo\ Qin},
       DOI = {10.1016/j.aim.2006.03.006},
       URL = {https://doi.org/10.1016/j.aim.2006.03.006},
}

@article {markman2002symplectic,
    AUTHOR = {Markman, Eyal},
     TITLE = {Generators of the cohomology ring of moduli spaces of sheaves
              on symplectic surfaces},
   JOURNAL = {J. Reine Angew. Math.},
  FJOURNAL = {Journal f\"{u}r die Reine und Angewandte Mathematik. [Crelle's
              Journal]},
    VOLUME = {544},
      YEAR = {2002},
     PAGES = {61--82},
      ISSN = {0075-4102,1435-5345},
   MRCLASS = {14D20 (14J28 14J60)},
  MRNUMBER = {1887889},
MRREVIEWER = {Jean-Marc\ Dr\'{e}zet},
       DOI = {10.1515/crll.2002.028},
       URL = {https://doi.org/10.1515/crll.2002.028},
}

@article {biquardboalch2004,
    AUTHOR = {Biquard, Olivier and Boalch, Philip},
     TITLE = {Wild non-abelian {H}odge theory on curves},
   JOURNAL = {Compos. Math.},
  FJOURNAL = {Compositio Mathematica},
    VOLUME = {140},
      YEAR = {2004},
    NUMBER = {1},
     PAGES = {179--204},
      ISSN = {0010-437X,1570-5846},
   MRCLASS = {53C07 (32G34 34M55)},
  MRNUMBER = {2004129},
MRREVIEWER = {Usha\ N.\ Bhosle},
       DOI = {10.1112/S0010437X03000010},
       URL = {https://doi.org/10.1112/S0010437X03000010},
}

@misc{boalch2012hyperkahler,
      title={Hyperkahler manifolds and nonabelian Hodge theory of (irregular) curves}, 
      author={Philip Boalch},
      year={2012},
      eprint={1203.6607},
      archivePrefix={arXiv},
      primaryClass={math.AG}
}

@unpublished{ShenMaulik2022,
    author = {Shen, Junliang and Maulik, Davesh},
    title = {The P=W conjecture for $GL_n$},
    note = {https://arxiv.org/abs/2209.02568}
}

@article {HauselRodriguez2013,
    AUTHOR = {Hausel, Tam\'{a}s and Rodriguez Villegas, Fernando},
     TITLE = {Cohomology of large semiprojective hyperk\"{a}hler varieties},
   JOURNAL = {Ast\'{e}risque},
  FJOURNAL = {Ast\'{e}risque},
    NUMBER = {370},
      YEAR = {2015},
     PAGES = {113--156},
      ISSN = {0303-1179,2492-5926},
      ISBN = {978-2-85629-806-0},
   MRCLASS = {53C26 (14C30 14D20 14L30 20G05 62E17)},
  MRNUMBER = {3364745},
MRREVIEWER = {Shintar\^{o}\ Kuroki},
}

@article {Yokogawa-compactification,
    AUTHOR = {Yokogawa, K\^{o}ji},
     TITLE = {Compactification of moduli of parabolic sheaves and moduli of
              parabolic {H}iggs sheaves},
   JOURNAL = {J. Math. Kyoto Univ.},
  FJOURNAL = {Journal of Mathematics of Kyoto University},
    VOLUME = {33},
      YEAR = {1993},
    NUMBER = {2},
     PAGES = {451--504},
      ISSN = {0023-608X},
   MRCLASS = {14D20 (14F05)},
  MRNUMBER = {1231753},
MRREVIEWER = {Yves\ Laszlo},
       DOI = {10.1215/kjm/1250519269},
       URL = {https://doi.org/10.1215/kjm/1250519269},
}

@article {Yokogawa-infinitesimal,
    AUTHOR = {Yokogawa, K\^{o}ji},
     TITLE = {Infinitesimal deformation of parabolic {H}iggs sheaves},
   JOURNAL = {Internat. J. Math.},
  FJOURNAL = {International Journal of Mathematics},
    VOLUME = {6},
      YEAR = {1995},
    NUMBER = {1},
     PAGES = {125--148},
      ISSN = {0129-167X,1793-6519},
   MRCLASS = {14F05 (14D15 14D20)},
  MRNUMBER = {1307307},
MRREVIEWER = {Nitin\ Nitsure},
       DOI = {10.1142/S0129167X95000092},
       URL = {https://doi.org/10.1142/S0129167X95000092},
}

@article {Maruyama-Yokogawa,
    AUTHOR = {Maruyama, M. and Yokogawa, K.},
     TITLE = {Moduli of parabolic stable sheaves},
   JOURNAL = {Math. Ann.},
  FJOURNAL = {Mathematische Annalen},
    VOLUME = {293},
      YEAR = {1992},
    NUMBER = {1},
     PAGES = {77--99},
      ISSN = {0025-5831,1432-1807},
   MRCLASS = {14D20 (14D22 14F05 14H60)},
  MRNUMBER = {1162674},
MRREVIEWER = {Nitin\ Nitsure},
       DOI = {10.1007/BF01444704},
       URL = {https://doi.org/10.1007/BF01444704},
}

@article {bayermacri,
    AUTHOR = {Bayer, Arend and Macr\`{i}, Emanuele},
     TITLE = {Projectivity and birational geometry of {B}ridgeland moduli
              spaces},
   JOURNAL = {J. Amer. Math. Soc.},
  FJOURNAL = {Journal of the American Mathematical Society},
    VOLUME = {27},
      YEAR = {2014},
    NUMBER = {3},
     PAGES = {707--752},
      ISSN = {0894-0347,1088-6834},
   MRCLASS = {14D20 (14E30 14F05 14J28)},
  MRNUMBER = {3194493},
MRREVIEWER = {P.\ E.\ Newstead},
       DOI = {10.1090/S0894-0347-2014-00790-6},
       URL = {https://doi.org/10.1090/S0894-0347-2014-00790-6},
}

@article {Matsuki-Wentworth,
    AUTHOR = {Matsuki, Kenji and Wentworth, Richard},
     TITLE = {Mumford-{T}haddeus principle on the moduli space of vector
              bundles on an algebraic surface},
   JOURNAL = {Internat. J. Math.},
  FJOURNAL = {International Journal of Mathematics},
    VOLUME = {8},
      YEAR = {1997},
    NUMBER = {1},
     PAGES = {97--148},
      ISSN = {0129-167X,1793-6519},
   MRCLASS = {14D20 (14F05)},
  MRNUMBER = {1433203},
MRREVIEWER = {Zhenbo\ Qin},
       DOI = {10.1142/S0129167X97000068},
       URL = {https://doi.org/10.1142/S0129167X97000068},
}

@article {Kuchle,
    AUTHOR = {K\"{u}chle, Oliver},
     TITLE = {Ample line bundles on blown up surfaces},
   JOURNAL = {Math. Ann.},
  FJOURNAL = {Mathematische Annalen},
    VOLUME = {304},
      YEAR = {1996},
    NUMBER = {1},
     PAGES = {151--155},
      ISSN = {0025-5831,1432-1807},
   MRCLASS = {14J25 (14C20 14E05)},
  MRNUMBER = {1367887},
       DOI = {10.1007/BF01446289},
       URL = {https://doi.org/10.1007/BF01446289},
}

@incollection {Kontsevich-Soibelman,
    AUTHOR = {Kontsevich, Maxim and Soibelman, Yan},
     TITLE = {Wall-crossing structures in {D}onaldson-{T}homas invariants,
              integrable systems and mirror symmetry},
 BOOKTITLE = {Homological mirror symmetry and tropical geometry},
    SERIES = {Lect. Notes Unione Mat. Ital.},
    VOLUME = {15},
     PAGES = {197--308},
 PUBLISHER = {Springer, Cham},
      YEAR = {2014},
   MRCLASS = {14N35 (14J33 53D37)},
  MRNUMBER = {3330788},
MRREVIEWER = {Victor\ Przyjalkowski},
       DOI = {10.1007/978-3-319-06514-4\_6},
       URL = {https://doi.org/10.1007/978-3-319-06514-4_6},
}

@article {inaba-saito,
    AUTHOR = {Inaba, Michi-aki and Saito, Masa-Hiko},
     TITLE = {Moduli of unramified irregular singular parabolic connections
              on a smooth projective curve},
   JOURNAL = {Kyoto J. Math.},
  FJOURNAL = {Kyoto Journal of Mathematics},
    VOLUME = {53},
      YEAR = {2013},
    NUMBER = {2},
     PAGES = {433--482},
      ISSN = {2156-2261,2154-3321},
   MRCLASS = {14D20 (34M55 34M56)},
  MRNUMBER = {3079310},
MRREVIEWER = {Fortun\'{e}\ Massamba},
       DOI = {10.1215/21562261-2081261},
       URL = {https://doi.org/10.1215/21562261-2081261},
}

@book {newstead,
    AUTHOR = {Newstead, P. E.},
     TITLE = {Introduction to moduli problems and orbit spaces},
    SERIES = {Tata Institute of Fundamental Research Lectures on Mathematics
              and Physics},
    VOLUME = {51},
 PUBLISHER = {Tata Institute of Fundamental Research, Bombay; Narosa
              Publishing House, New Delhi},
      YEAR = {1978},
     PAGES = {vi+183},
      ISBN = {0-387-08851-2},
   MRCLASS = {14-02 (14D20)},
  MRNUMBER = {546290},
MRREVIEWER = {G.\ Horrocks},
}

@article {mcgerty-nevins,
    AUTHOR = {McGerty, Kevin and Nevins, Thomas},
     TITLE = {Kirwan surjectivity for quiver varieties},
   JOURNAL = {Invent. Math.},
  FJOURNAL = {Inventiones Mathematicae},
    VOLUME = {212},
      YEAR = {2018},
    NUMBER = {1},
     PAGES = {161--187},
      ISSN = {0020-9910,1432-1297},
   MRCLASS = {14D20 (14F05 14F43 14L24 16G20)},
  MRNUMBER = {3773791},
MRREVIEWER = {P.\ E.\ Newstead},
       DOI = {10.1007/s00222-017-0765-x},
       URL = {https://doi.org/10.1007/s00222-017-0765-x},
}

@article {atiyah-bott,
    AUTHOR = {Atiyah, M. F. and Bott, R.},
     TITLE = {The {Y}ang-{M}ills equations over {R}iemann surfaces},
   JOURNAL = {Philos. Trans. Roy. Soc. London Ser. A},
  FJOURNAL = {Philosophical Transactions of the Royal Society of London.
              Series A. Mathematical and Physical Sciences},
    VOLUME = {308},
      YEAR = {1983},
    NUMBER = {1505},
     PAGES = {523--615},
      ISSN = {0080-4614},
   MRCLASS = {14F05 (14D22 32G13 32L05 53C05 58E05 58E20 81E13)},
  MRNUMBER = {702806},
MRREVIEWER = {Martin\ A.\ Guest},
       DOI = {10.1098/rsta.1983.0017},
       URL = {https://doi.org/10.1098/rsta.1983.0017},
}

@incollection {beauville,
    AUTHOR = {Beauville, Arnaud},
     TITLE = {Sur la cohomologie de certains espaces de modules de
              fibr\'{e}s vectoriels},
 BOOKTITLE = {Geometry and analysis ({B}ombay, 1992)},
     PAGES = {37--40},
 PUBLISHER = {Tata Inst. Fund. Res., Bombay},
      YEAR = {1995},
      ISBN = {0-19-563740-2},
   MRCLASS = {14D20 (14H60)},
  MRNUMBER = {1351502},
MRREVIEWER = {Daniel\ Huybrechts},
}

@article {ellingsrud-stromme,
    AUTHOR = {Ellingsrud, Geir and Stromme, Stein Arild},
     TITLE = {Towards the {C}how ring of the {H}ilbert scheme of {${\mathbf P}^2$}},
   JOURNAL = {J. Reine Angew. Math.},
  FJOURNAL = {Journal f\"{u}r die Reine und Angewandte Mathematik. [Crelle's
              Journal]},
    VOLUME = {441},
      YEAR = {1993},
     PAGES = {33--44},
      ISSN = {0075-4102,1435-5345},
   MRCLASS = {14C05 (14D20)},
  MRNUMBER = {1228610},
MRREVIEWER = {Vasile\ Br\^{i}nz\u{a}nescu},
}

@article {biswas-raghavendra,
    AUTHOR = {Biswas, Indranil and Raghavendra, N.},
     TITLE = {Canonical generators of the cohomology of moduli of parabolic
              bundles on curves},
   JOURNAL = {Math. Ann.},
  FJOURNAL = {Mathematische Annalen},
    VOLUME = {306},
      YEAR = {1996},
    NUMBER = {1},
     PAGES = {1--14},
      ISSN = {0025-5831,1432-1807},
   MRCLASS = {14D20},
  MRNUMBER = {1405316},
MRREVIEWER = {Vasily\ A.\ Chernecky},
       DOI = {10.1007/BF01445239},
       URL = {https://doi.org/10.1007/BF01445239},
}

@book {huybrechts-lehn,
    AUTHOR = {Huybrechts, Daniel and Lehn, Manfred},
     TITLE = {The geometry of moduli spaces of sheaves},
    SERIES = {Aspects of Mathematics},
    VOLUME = {E31},
 PUBLISHER = {Friedr. Vieweg \& Sohn, Braunschweig},
      YEAR = {1997},
     PAGES = {xiv+269},
      ISBN = {3-528-06907-4},
   MRCLASS = {14D20 (14F05)},
  MRNUMBER = {1450870},
MRREVIEWER = {Jean-Marc\ Dr\'{e}zet},
       DOI = {10.1007/978-3-663-11624-0},
       URL = {https://doi.org/10.1007/978-3-663-11624-0},
}

@misc{hausel2022pw,
      title={$P=W$ via $H_2$}, 
      author={Tamas Hausel and Anton Mellit and Alexandre Minets and Olivier Schiffmann},
      year={2022},
      eprint={2209.05429},
      archivePrefix={arXiv},
      primaryClass={math.AG}
}

@misc{maulik2022pw,
      title={The $P=W$ conjecture for $\mathrm{GL}_n$}, 
      author={Davesh Maulik and Junliang Shen},
      year={2022},
      eprint={2209.02568},
      archivePrefix={arXiv},
      primaryClass={math.AG}
}

@misc{maulik2023perverse,
      title={Perverse filtrations and Fourier transforms}, 
      author={Davesh Maulik and Junliang Shen and Qizheng Yin},
      year={2023},
      eprint={2308.13160},
      archivePrefix={arXiv},
      primaryClass={math.AG}
}

@article {sabbah,
    AUTHOR = {Sabbah, Claude},
     TITLE = {Harmonic metrics and connections with irregular singularities},
   JOURNAL = {Ann. Inst. Fourier (Grenoble)},
  FJOURNAL = {Universit\'{e} de Grenoble. Annales de l'Institut Fourier},
    VOLUME = {49},
      YEAR = {1999},
    NUMBER = {4},
     PAGES = {1265--1291},
      ISSN = {0373-0956,1777-5310},
   MRCLASS = {32S40 (32L10 32S60)},
  MRNUMBER = {1703088},
       URL = {http://www.numdam.org/item?id=AIF_1999__49_4_1265_0},
}

@article {szabo21pw,
    AUTHOR = {Szab\'{o}, Szil\'{a}rd},
     TITLE = {Perversity equals weight for {P}ainlev\'{e} spaces},
   JOURNAL = {Adv. Math.},
  FJOURNAL = {Advances in Mathematics},
    VOLUME = {383},
      YEAR = {2021},
     PAGES = {Paper No. 107667, 45},
      ISSN = {0001-8708,1090-2082},
   MRCLASS = {14H60 (14D20)},
  MRNUMBER = {4232540},
MRREVIEWER = {G.\ Ruzza},
       DOI = {10.1016/j.aim.2021.107667},
       URL = {https://doi.org/10.1016/j.aim.2021.107667},
}

@article {Szabo23pw,
    AUTHOR = {Szab\'{o}, Szil\'{a}rd},
     TITLE = {Hitchin {WKB}-problem and {$P=W$} conjecture in lowest degree
              for rank 2 over the 5-punctured sphere},
   JOURNAL = {Q. J. Math.},
  FJOURNAL = {The Quarterly Journal of Mathematics},
    VOLUME = {74},
      YEAR = {2023},
    NUMBER = {2},
     PAGES = {687--746},
      ISSN = {0033-5606,1464-3847},
   MRCLASS = {14H60 (14N35 14T20)},
  MRNUMBER = {4596216},
       DOI = {10.1093/qmath/haac037},
       URL = {https://doi.org/10.1093/qmath/haac037},
}

@article {shen-zhang21pw,
    AUTHOR = {Shen, Junliang and Zhang, Zili},
     TITLE = {Perverse filtrations, {H}ilbert schemes, and the {$P=W$}
              conjecture for parabolic {H}iggs bundles},
   JOURNAL = {Algebr. Geom.},
  FJOURNAL = {Algebraic Geometry},
    VOLUME = {8},
      YEAR = {2021},
    NUMBER = {4},
     PAGES = {465--489},
      ISSN = {2313-1691,2214-2584},
   MRCLASS = {14F08 (14C05 14D20 14H60)},
  MRNUMBER = {4371537},
MRREVIEWER = {Camilla\ Felisetti},
       DOI = {10.14231/ag-2021-014},
       URL = {https://doi.org/10.14231/ag-2021-014},
}

@article {simpson-noncompact,
    AUTHOR = {Simpson, Carlos T.},
     TITLE = {Harmonic bundles on noncompact curves},
   JOURNAL = {J. Amer. Math. Soc.},
  FJOURNAL = {Journal of the American Mathematical Society},
    VOLUME = {3},
      YEAR = {1990},
    NUMBER = {3},
     PAGES = {713--770},
      ISSN = {0894-0347,1088-6834},
   MRCLASS = {58E20 (14C30 14H60 32G20)},
  MRNUMBER = {1040197},
MRREVIEWER = {N.\ J.\ Hitchin},
       DOI = {10.2307/1990935},
       URL = {https://doi.org/10.2307/1990935},
}

@article {dCMH12,
    AUTHOR = {de Cataldo, Mark Andrea A. and Hausel, Tam\'{a}s and
              Migliorini, Luca},
     TITLE = {Topology of {H}itchin systems and {H}odge theory of character
              varieties: the case {$A_1$}},
   JOURNAL = {Ann. of Math. (2)},
  FJOURNAL = {Annals of Mathematics. Second Series},
    VOLUME = {175},
      YEAR = {2012},
    NUMBER = {3},
     PAGES = {1329--1407},
      ISSN = {0003-486X,1939-8980},
   MRCLASS = {14D20 (53D30)},
  MRNUMBER = {2912707},
MRREVIEWER = {Sean\ Lawton},
       DOI = {10.4007/annals.2012.175.3.7},
       URL = {https://doi.org/10.4007/annals.2012.175.3.7},
}

@article {yun-global-springer,
    AUTHOR = {Yun, Zhiwei},
     TITLE = {Global {S}pringer theory},
   JOURNAL = {Adv. Math.},
  FJOURNAL = {Advances in Mathematics},
    VOLUME = {228},
      YEAR = {2011},
    NUMBER = {1},
     PAGES = {266--328},
      ISSN = {0001-8708,1090-2082},
   MRCLASS = {14F05 (14F20 14M15 17B45 20C08 20G15 20G44)},
  MRNUMBER = {2822234},
MRREVIEWER = {J.\ Matthew\ Douglass},
       DOI = {10.1016/j.aim.2011.05.012},
       URL = {https://doi.org/10.1016/j.aim.2011.05.012},
}

@article {komyo,
    AUTHOR = {Komyo, Arata},
     TITLE = {Mixed {H}odge structures of the moduli spaces of parabolic
              connections},
   JOURNAL = {Nagoya Math. J.},
  FJOURNAL = {Nagoya Mathematical Journal},
    VOLUME = {225},
      YEAR = {2017},
     PAGES = {185--206},
      ISSN = {0027-7630,2152-6842},
   MRCLASS = {14D22 (14C30 14D20)},
  MRNUMBER = {3624424},
MRREVIEWER = {Ajneet\ S.\ Dhillon},
       DOI = {10.1017/nmj.2016.38},
       URL = {https://doi.org/10.1017/nmj.2016.38},
}

@misc{oblomkov2017cohomology,
      title={The cohomology ring of certain compactified Jacobians}, 
      author={Alexei Oblomkov and Zhiwei Yun},
      year={2017},
      eprint={1710.05391},
      archivePrefix={arXiv},
      primaryClass={math.AG}
}

@article{Szab2017TheBG,
  title={The birational geometry of unramified irregular Higgs bundles on curves},
  author={Szil{\'a}rd Szab{\'o}},
  journal={International Journal of Mathematics},
  year={2017},
  volume={28},
  pages={1750045},
  url={https://api.semanticscholar.org/CorpusID:125343111}
}

\end{document}